\documentclass[11pt]{article}

\usepackage{color}
\usepackage{subfigure}
\usepackage{authblk}

\usepackage{amsmath,amsthm,amssymb}
\usepackage{color}
\usepackage{graphicx}
\usepackage{subfigure}
\usepackage{psfrag}

\usepackage{amsmath,amsthm,amssymb}
\usepackage{textcomp}
\usepackage{authblk}
\usepackage[normalem]{ulem}
\usepackage{shadow} %\fbox{}, \shabox{}
\usepackage{fancybox}  %\shadowbox{}, \doublebox{}, \ovalbox{}, \Ovalbox{}
\newcommand{\supp}{\mathop{\mathrm{supp}}}

\newcommand{\myspan}{\mathop{\mathrm{span}}}
\newcommand{\trunc}{\mathop{\mathrm{trunc}}}
\newcommand{\Trunc}{\mathop{\mathrm{Trunc}}}
\newcommand{\myvert}[1]{\!\left.\rule{0ex}{2ex}\right|_{#1}}

\newcommand{\osc}{\mathrm{osc}}

\let\hat\widehat

\definecolor{dgreen}{rgb}{0.5,0.2,1}
\definecolor{gold}{rgb}{.7,.5,0}
\definecolor{dred}{rgb}{0.92,0,0}
\definecolor{dgreen}{rgb}{0,0.6,0}
\def\B {\color{black}}
\def\Rd{\color{black}}
\def\Gd{\color{black}}

\newcommand{\VV}{\mathbb{V}}
\renewcommand{\SS}{\mathbb{S}}
\newcommand{\cB}{\mathcal{B}}
\newcommand{\cG}{\mathcal{G}}
\newcommand{\cM}{\mathcal{M}}
\newcommand{\cN}{\mathcal{N}}
\newcommand{\cQ}{\mathcal{Q}}

\newcommand{\bx}{\mathbf{x}}
\let\les\lesssim

\newcommand{\bn}{{\bf n}}

\newcommand{\bF}{{\bf F}}

\newcommand{\bxi}{\boldsymbol{\xi}}

\def\bx{{\bf x}}

\renewcommand{\div}{\textrm{div}}
\newcommand{\bA}{\textbf{A}} 

\newtheorem{thm}{Theorem}
\newtheorem{dfn}[thm]{Definition}
\newtheorem{rmk}[thm]{Remark}
\newtheorem{lma}[thm]{Lemma}
\newtheorem{crl}[thm]{Corollary}
\newtheorem{prn}[thm]{Proposition}

\newtheorem{exm}[thm]{Example}

\newenvironment{proofof}[1]{\begin{trivlist}
                      \item[]\hspace{0cm}{\bf Proof~of~{#1}.}}
                      {\hfill $\qed$
                    \end{trivlist}}
\def\M3AS{Math.\ Models\ Methods\ Appl.\ Sci.}

\begin{document}

\markboth{A.~Buffa \& C.~Giannelli}{Adaptive isogeometric methods with hierarchical splines: error estimator and convergence}

%%%%%%%%%%%%%%%%%%% Publisher's Area please ignore %%%%%%%%%%%%%%%%%%%%%%%
%
%\catchline{}{}{}{}{}
%
%%%%%%%%%%%%%%%%%%%%%%%%%%%%%%%%%%%%%%%%%%%%%%%%%%%%%%%%%%%%%%%%%%%%%%%%%%

\title{\normalsize \bf ADAPTIVE ISOGEOMETRIC METHODS WITH HIERARCHICAL SPLINES:\\ ERROR ESTIMATOR AND CONVERGENCE
%\footnote{For the title, try not to 
%use more than 3 lines. Typeset the title in 10 pt 
%Times roman, uppercase and boldface.} 
}

\author{\footnotesize ANNALISA BUFFA\footnote{Istituto di Matematica Applicata e Tecnologie Informatiche `E.~Magenes' del CNR,  via Ferrata 1, 27100 Pavia Italy. E-mail address: annalisa@imati.cnr.it} \, and CARLOTTA GIANNELLI\footnote{Istituto Nazionale di Alta Matematica, 
Unit\`a di Ricerca di Firenze c/o DiMaI `U.~Dini', Universit\`a di Firenze, viale Morgagni 67a, 50134 Firenze, Italy. E-mail address:  carlotta.giannelli@unifi.it}} 

%\affil[*]{\scriptsize Istituto di Matematica Applicata e Tecnologie Informatiche `E.~Magenes' del CNR,  
%via Ferrata 1, 27100 Pavia Italy\\
%annalisa@imati.cnr.it}
%\footnote{State completely without abbreviations, the
%affiliation and mailing address, including country. Typeset in 8 pt
%Times italic.}\\
%first\_author@university.edu}

%\author{\footnotesize CARLOTTA GIANNELLI}
%\footnote{Istituto Nazionale di Alta Matematica, 
%Unit\`a di Ricerca di Firenze c/o DiMaI `U.~Dini', INdAM c/o Dipartimento di Matematica e Informatica ``U.~Dini,'' Universit\`a di Firenze, viale Morgagni 67a, 50134 Firenze (carlotta.giannelli@unifi.it)}}

%\affil[**]{\footnotesize Istituto Nazionale di Alta Matematica, 
%Unit\`a di Ricerca di Firenze c/o DiMaI %`U.~Dini'\\
%%INdAM c/o Dipartimento di Matematica e Informatica ``U.~Dini,'' 
%%Universit\`a di Firenze, 
%%viale Morgagni 67a, 50134 Firenze, Italy.\\
%carlotta.giannelli@unifi.it}

\date{}

\maketitle

%\begin{history}
%\received{(Day Month Year)}
%\revised{(Day Month Year)}
%%\accepted{(Day Month Year)}
%\comby{(xxxxxxxxxx)}
%\end{history}

\begin{abstract}\noindent
The problem of developing an adaptive isogeometric method (AIGM) for solving elliptic second-order partial differential equations with truncated hierarchical B-splines of arbitrary degree and different order of continuity is addressed. The adaptivity analysis holds in any space dimensions.
We consider a simple residual-type error estimator for which we provide a posteriori upper and lower bound in terms of local error indicators, taking also into account the critical role of oscillations as in a standard adaptive finite element setting. The error estimates are properly combined with a simple marking strategy to define a sequence of  admissible locally refined meshes and corresponding approximate solutions. The design of a refine module that preserves the admissibility of the hierarchical mesh configuration between two consectutive steps of the adaptive loop is presented. The contraction property of the quasi-error, given by the sum of the energy error and the scaled error estimator, leads to the convergence proof of the AIGM.
%Abstract is required and should summarize, in less 
%than 300 words, the context, content and conclusions of the
%paper. It should not contain any references or displayed
%equations.  Textwidth of the abstract should be 4.5 inches.
\end{abstract}

%\keywords{isogeometric analysis; hierarchical splines; adaptivity.}

%\ccode{AMS Subject Classification: 41A15,65D07,65N12,65N30}
%41A15 Spline approximation
%65D07 Splines
%65N12 Stability and convergence of numerical methods
%65N30 Finite elements, Rayleigh-Ritz and Galerkin methods, finite methods

%Error analysis and interval analysis
%65G99 None of the above, but in this section
%--------------------------------------------------------------------------------
\section{Introduction}
\label{sec:aloop}
%--------------------------------------------------------------------------------

The definition of adaptive schemes  that provide local mesh refinement is an active area of research in the context of isogeometric analysis \cite{cottrell2009,hughes2005}, an emerging paradigm for the solution of partial differential equations which combines and extends finite 
element techniques with computer aided design (CAD) methods related to spline models. Since the CAD standard for spline representation in a multivariate setting relies on tensor-product B-splines, e.g.~see \cite{deboor2001,schumaker2007}, an adaptive isogeometric model necessarily requires suitable extensions of the B-spline model that  give the possibility to relax \B the rigidity of the tensor-product structure by allowing hanging nodes.

There are a few different frameworks for the definition of splines on rectangular tiling with hanging nodes. We mention here  T-splines \cite{sederberg2004,sederberg2003} that have been used in the context of isogeometric analysis in the pioneering papers  
\cite{bazilevs2010,doerfel2010},
and their  analysis-suitable \cite{scott2011b} or dual-compatible 
\cite{daveiga2012,daveiga2013} versions.  Other possibilities are offered by  polynomial  splines over  ({\Rd hierarchical}) T-meshes
\cite{deng2006,deng2008} or LR-splines \cite{dokken2013,bressan2013}, that have been tested within an isogeometric framework in  \cite{nguyen-thanh2011} and \cite{johannessen2014}, respectively.  

Finally, hierarchical splines based on the construction presented in \cite{kraft1997} is one of the most promising approach.  This is also due to the fact that their construction and properties are closely related to the ones of hierarchical finite elements.  Hierarchical B-spline constructions and their use, both as an adaptive modeling tool, as well as a framework for isogeometric analysis that provides local refinement possibilities, has been recently investigated in a number of papers, see e.g.~\cite{vuong2011,giannelli2012,giannelli2014,kiss2014b}.

In the present paper we aim at defining and studying an \emph{adaptive isogeometric method (AIGM) based on hierarchical splines}.  The choice,   among the adaptive spline models mentioned above, of the hierarchical setting have a twofold motivation. On the one hand, it is a natural extension of the B-spline model that is able to preserve many key properties directly by construction, and the refinement rules are simple and straightforward. In addition, \B although the type of refinement they allow is more restrictive than other solutions\B, the locally structured hierarchical approach allows to defines an effective automatically-driven refinement strategy that, in turns, can be used to design a fully adaptive method.

We consider the simple elliptic model problem:
\begin{equation}
\label{eq:mp}
-\div (\bA \nabla u) = f \quad\text{in}\; \Omega, 
\qquad 
u\myvert{\partial\Omega} = 0,
\end{equation}
where $\Omega\subset\mathbb{R}^{d}$, $d\ge 1$, is a bounded domain with Lipschitz boundary $\partial\Omega$, and $f$ is any square integrable function and 
\begin{equation}
  \label{eq:5}
  \forall \bx\in \Omega,\ \bxi\in \mathbb{R}^d  \;  
  \eta_1 |\bxi|^2 \leq \bA(x)\bxi\cdot\bxi \; \text{ and } \; |\bA(\bx)\bxi|\leq \eta_2 |\bxi|
 \end{equation}
with $0 < \eta_1\leq \eta_2$. 

By closely following the framework of adaptive finite elements --- see e.g., the recent reviews in \cite{nsv2009,NochettoCIME} and references therein --- for elliptic partial differential equations, we aim at designing and analyse  the four  blocks in the following flowchart  associated to an AIGM. 

\begin{center}
\fbox{SOLVE} $\rightarrow$ \fbox{ESTIMATE} $\rightarrow$ \fbox{MARK} $\rightarrow$ \fbox{REFINE}
%\shadowbox{SOLVE} $\rightarrow$ \shadowbox{ESTIMATE} $\rightarrow$ \shadowbox{MARK} $\rightarrow$ \shadowbox{REFINE} 
%\shabox{SOLVE} $\rightarrow$ \shabox{ESTIMATE} $\rightarrow$ \shabox{MARK} $\rightarrow$ \shabox{REFINE} 
\end{center}

At our best knowledge, all  previous works on error estimators in isogeometric analysis were mainly devoted to numerical experiments with some goal--oriented error estimators based on auxiliary global refinement steps \cite{zv2011,ds2012,kvvb2014}.

Our choices for the different steps of the adaptive loop may be detailed as follows.

\begin{itemize}

\item[\fbox{SOLVE}] We want  to solve  problem \eqref{eq:mp}  with hierarchical spline spaces. To this aim, we define a family of \emph{admissible} hierarchical meshes, which uses the concept of truncated basis \cite{giannelli2012}, and we consider the Galerkin method on these spaces. Admissibility is related to the number of levels which are present (with non zero basis functions) on an element, and it is a fundamental assumption in our theory.  

\item[\fbox{ESTIMATE}] We define residual based error estimator for our problem. Thanks to the regularity of splines, such an estimator  reduces to the $L^2$-norm of the element-by-element residual suitably weighted with the mesh size. We prove that this estimator is \emph{reliable}, i.e., it is an upper bound for the error, and \emph{efficient}, i.e., it is a lower bound of the error (up to oscillations).

\item[\fbox{MARK}] We adopt the D\"orfler marking strategy 
\cite{dorfler1996},  namely we mark for refinement all elements with largest error indicator until a certain fixed percentage of the total error indicator is taken into account  by the set of marked elements. 

\item[\fbox{REFINE}] A refinement procedure constructs the refined mesh starting from the set of marked elements,  by following the structure of the recursive refine module generally considered in adaptive finite elements, see e.g.~\cite{morin2001,morin2002}.  We construct this routine so that the admissibility of the refined mesh is  preserved between two consecutive iterations of the adaptive loop. 
\end{itemize}
In general, the refinement procedure identifies the mesh with an increased level of resolution for the next iteration by refinining not only the marked elements, but also a suitable set of elements in their neighbourhood,  analogously to the concept of \emph{refinement patches} in an adaptive finite element method. This allows to construct a mesh that preserves a certain class of admissibility. The refinement mechanism is similar to the strategy adopted to bound the number of hanging nodes per side in the refinement of quadrilateral meshes for finite elements \cite{bonito2010}, and  is also related  to the properties of the domain partitions created by the bisection rule that are needed to prove quasi-optimality of adaptive finite element methods 
\cite{binev2004a,cascon2008,stevenson2007,stevenson2008}. 

In the present paper we start the numerical analysis of our AIGM method and we provide a convergence result together with the contraction of the quasi-error (i.e., the sum of the error and the error indicator), while the complexity of the refine routine, together with quasi-interpolation operators and optimality of the AIGM, is left to the companion paper \cite{buffa2015b}. 

% Finally, it should be said that specific bounds for the number of the non-zero basis functions on any mesh element plays a key role for the development of our adaptivity theory. 

The paper is organized as follows. Some preliminary aspects of hierarchial tensor-product B-spline constructions are reviewed in Section~\ref{sec:hspaces} together with the definition of {\Rd truncated hierarchical B-splines (THB-splines)}  and related properties, before introducing the notion of  (strictly) admissible meshes.
The module SOLVE and {\Rd ESTIMATE} of the adaptive isogeometric method are discussed in Sections~\ref{sec:solve} and \ref{sec:estimate} including {\Rd an a} posteriori error analysis in terms of both upper and lower bound for the energy error. Section~\ref{sec:mark&refine} recalls a well-known marking strategy and introduces a refinement strategy that preserves the class of admissibility during the iterative loop --- module MARK and REFINE. Finally, Section~\ref{sec:closure} concludes the paper by summarizing the key results of the present study, and outlines the spirit of our companion paper \cite{buffa2015b}. 
%--------------------------------------------------------------------------------
%------------------------------------------------------------------
\section{Hierarchical spline spaces}
\label{sec:hspaces}
%------------------------------------------------------------------
We start by considering the hierarchical approach to adaptive mesh refinement, as natural  extension of the standard tensor-product B-spline model in a general multivariate setting. In particular, we focus on the truncated hierachical B-spline basis, since it allows us to identify a certain class of admissible mesh configurations.
%------------------------------------------------------------------
\subsection{Preliminaries: B-spline hierarchies}
\label{sec:pre}
%------------------------------------------------------------------
Hierarchical B-spline spaces are constructed by considering a hierarchy of $N$ tensor-product $d$-variate spline spaces $V^0\subset V^1\subset \ldots...\subset V^{N-1}$ defined on a {\Rd closed hyper-rectangle} $D$ in $\mathbb{R}^d$ together with a hierarchy of domains $\hat{\Omega}^0\supseteq\hat{\Omega}^1\supseteq\ldots\supseteq\hat{\Omega}^{N-1}$, that are {\Rd closed} subsets of $D$. The \emph{depth} of the subdomain hierarchy is represented by the integer $N$, and we assume $\hat{\Omega}^N=\emptyset$.

For each level $\ell$, with $\ell=0,1,\dots,N-1$, the multivariate spline space $V^\ell$ is spanned by the tensor-product B-spline basis $\hat{{\cal B}}^\ell$ of degree $\mathbf{p}=(p_1,\ldots,p_d)$ defined on {\Rd a given tensor-product} grid $\hat{G}^\ell$. The (non-empty) quadrilateral elements (or cells) $\hat{Q}$ of $\hat{G}^\ell$ are the Cartesian product of $d$ open intervals between adjacent grid values. For any coordinate direction $i$, for $i=1,\ldots,d$ the knot sequences associated to the grids at the different levels contain non-decreasing real numbers so that each grid value appears in the knot vector as many times as specified by a certain multiplicity. At any level $\ell$, i.e., for the case of standard tensor-product B-splines, the multiplicity of each knot may vary between one (single knots) and $p_i$ or $p_{i+1}$ for the case of continuous and discontinuos functions, respectively. In order to guarantee the nested nature of the spline spaces $V^\ell\subset V^{\ell+1}$, we require that every knot of level $\ell-1$ is also present at level $\ell$ at least with the same multiplicity in the corresponding coordinate direction.
 
From the classical spline theory, it is known that B-splines are locally linear independent, they are non-negative, they have local support, and form a partition of unity \cite{deboor2001,schumaker2007}. Moreover, there exists a two-scale relation between adjacent bases in the hierarchy so that any function $s\in V^\ell\subset V^{\ell+1}$ can be expressed as
\begin{equation}\label{eq:2scale}
s=\sum_{\hat{\beta}\in \hat{{\cal B}}^{\ell+1}} c_{\hat{\beta}}^{\ell+1}(s) \hat{\beta},
\end{equation}
in terms of {\Rd the} coefficients $c_{\hat{\beta}}^{\ell+1}$. 

%Each domain $\Omega^{\ell}$ is defined as a collection ${\cal G}^\ell=G^\ell\cap\Omega^\ell$ of elements $Q$ with respect to the tensor-product grid $G^\ell$. 
%Each tensor-product grid $\hat{G}^\ell$ defines a subdivision of the domain $\hat{\Omega}^\ell$ into a number of quadrilateral elements $\hat{Q}$ of level $\ell$,
{\Rd The domain ${\Omega}^\ell$ is defined as the union of the closure of elements of ${\hat G}^{\ell-1}$, namely}
\[
{\Rd {\hat{\Omega}}}^\ell=
\bigcup\left\{
\overline{\hat{Q}} {\Rd \,: \hat{Q}\in \hat{G}^{\ell-1}}
%{\Rd \,\wedge\,} \hat{Q}\subseteq\hat{\Omega}^\ell
\right\}.
\]
{\Rd An element $\hat{Q}$ of level $\ell$ is \emph{active} if $\hat{Q}\subset\hat{\Omega}^\ell$ and any $\hat{Q}^*$ of level $\ell^*>\ell$ which belongs to any $\hat{\Omega}^{\ell+1}, \ldots,\hat{\Omega}^{N-1}$ is not a subset of $\hat{Q}$.} We denote the collection of active elements of level $\ell$ as
\begin{equation}\label{eq:active}
\hat{{\cal G}}^\ell {\Rd\, :=\,} 
\left\{\hat{Q}\in \hat{G}^\ell : \hat{Q} {\Rd \,\subset\,} \hat{\Omega}^\ell \wedge
\nexists\; \hat{Q}^* {\Rd \in {\hat G}^{\ell^*},\, \ell^*>\ell : Q^*} \subset{\hat{\Omega}}^{\ell^*} {\Rd \,\wedge\,} \hat{Q}^* \subset \hat{Q}\right\}.
\end{equation}
Let {\Rd $\hat{\cal Q}$} be the
 \emph{mesh} composed by taking the active elements $Q$ at any hierarchical level, namely
\begin{equation}\label{eq:mesh}
\hat{{\cal Q}} {\Rd\,:=\,} \left\{
\hat{Q}\in \hat{{\cal G}}^\ell, \, {\Rd \ell } =0,\ldots,N-1
\right\}.
\end{equation}
For any $\hat{Q}\in\hat{{\cal Q}}$, we define $h_{\hat{Q}} {\Rd\, :=\,} |\hat{Q}|^{1/d}$.
A mesh $\hat{{\cal Q}}^*$ is a refinement of $\hat{{\cal Q}}$ if each element $\hat{Q}^*\in\hat{{\cal Q}}^*$ either also belongs to $\hat{{\cal Q}}$ or is obtained by splitting $\hat{Q}\in\hat{{\cal Q}}$ in $q^d$ elements via ``$q$-adic'' refinement, for some integer $q\ge 2$. The refinement relation between $\hat{\cQ}$ and $\hat{\cQ}^*$ will be indicated as $\hat{\cQ}^*\succeq\hat{\cQ}$.
{
In particular, we will consider the case of standard dyadic refinement with $q=2$.}
 
A basis for the hierarchical B-spline space can be constructed by a suitable selection of \emph{active} basis functions at different level of details according to the following definition, see also \cite{kraft1997,vuong2011}.
\begin{dfn}\label{dfn:hb}
The hierarchical {\Rd B-spline} (HB-spline) basis $\hat{{\cal H}}$ with respect to the mesh $\hat{\cal Q}$ is defined as
\begin{equation*}
\hat{{\cal H}}(\hat{{\cal Q}}) {\Rd\,:=\,} %\bigcup_{\ell=0,\ldots,N-1}
\left\{
\hat{\beta}\in\hat{{\cal B}}^\ell : 
\supp \hat{\beta} \subseteq\hat{\Omega}^\ell \wedge 
\supp \hat{\beta}\not\subseteq \hat{\Omega}^{\ell+1},
\, {\Rd \ell} =0,\ldots,N-1
\right\},
\end{equation*}
where $\supp \hat{\beta}$ denotes the intersection of the support of $\beta$ with $\hat{\Omega}^0$.
\end{dfn}

\begin{rmk}\label{rmk:pkref}
Note that the hierarchical approach is not confined to dyadic or $q$-adic (uniform) refinement, but it can also handle different kind of mesh refinements,
 including non-uniform configurations. 
In addition, by assuming that the degrees may increase (but not decrease) moving from one level to the subsequent in the hierarchy, nested sequence of tensor-product spline spaces can be also considered in the context of $p$- (and $k$-) refinement.
\end{rmk}
%$2\le q \le \max(p_1,\ldots,p_d)$
%\footnote{In this case, the grids should be nested in the following sense: every knot al level $\ell-1$ is also present at level $\ell$ at least with the same multiplicity plus the increase of the multiplicity in the corresponding coordinate direction.}
%------------------------------------------------------------------
\subsection{The truncated basis}
\label{sec:thb}
%------------------------------------------------------------------
We define the truncation of a function $\hat{s} \in V^\ell$ with respect to $\hat{\cal B}^{\ell+1}$ as the contributions in \eqref{eq:2scale} of only basis functions in $\hat{\cal B}^{\ell+1}$ that are \emph{passive}, i.e., not included in the hierarchical B-spline basis $\hat{{\cal H}}(\hat{\cQ})$. More precisely,
\begin{equation}\label{eq:trunc}
{\trunc}^{\ell+1} \hat{s} {\Rd\,:=\,} \sum_{\hat{\beta}\in \hat{\cB}^{\ell+1}, \, \supp\hat{\beta}\not\subseteq\hat{\Omega}^{\ell+1}} c_{\hat{\beta}}^{\ell+1}(s) \hat{\beta},
\end{equation}
where $c_{\hat{\beta}}^{\ell+1}(s)$ is the coefficient of the function $s$ with respect to the basis element $\hat{\beta}$ %\in \hat{\cB}^{\ell+1}$
at level $\ell+1$ of the B-spline refinement rule \eqref{eq:2scale}. By recursively applying the truncation to the HB-splines introduced in Definition~\ref{dfn:hb}, we can construct a different hierarchical basis 
\cite{giannelli2012}.

\begin{dfn}\label{dfn:thb}
The truncated hierarchical {\Rd B-spline} (THB-spline) basis $\hat{{\cal T}}$ with respect to the mesh $\hat{{\cal Q}}$ is defined as
\begin{equation*}
\hat{{\cal T}}(\hat{{\cal Q}}) {\Rd\,:=\,} %\bigcup_{\ell=0,\ldots,N-2}
\left\{
{\Rd {\Trunc}^{\ell+1}}\,\hat{\beta}:\hat{\beta}\in\hat{{\cal B}}^\ell
\cap\hat{{\cal H}}(\hat{{\cal Q}}),\,  {\Rd \ell} =0,\ldots,N-1\right\}, 
%\bigcup\left\{\beta:\beta\in{\cal B}^{N-1}\cap{\cal H}({\cal Q})\right\},
\end{equation*}
where $ {\Rd {\Trunc}^{\ell+1}}\,\hat{\beta} {\Rd\,:=\,} {\trunc}^{N-1}({\trunc}^{N-2}(\ldots ({\trunc}^{\ell+1}(\hat{\beta}))\dots))$, for any $\hat{\beta}\in\hat{{\cal B}}^\ell\cap\hat{{\cal H}}(\hat{{\cal Q}})$.
\end{dfn}

The \emph{level} of a truncated B-spline $\hat{\tau}\in\hat{{\cal T}}(\hat{{\cal Q}})$ is the level of the B-spline from which $\hat{\tau}$ is derived according to the iterative truncation mechanism introduced in Definition~\ref{dfn:thb}.
For simplicity, we will denote $\hat{{\cal H}}=\hat{{\cal H}}(\hat{{\cal Q}})$, $\hat{{\cal T}} = \hat{{\cal T}}(\hat{{\cal Q}})$ when there will be no ambiguity in the text. 
%------------------------------------------------------------------
\subsection{Properties of THB-splines}
\label{sec:pro}
%------------------------------------------------------------------
The truncated basis $\hat{{\cal T}}$ not only spans the same hierarchical space of classical HB-splines, namely 
\begin{itemize}
\item[(i)] $\myspan{\hat{\cal T}}=\myspan{\hat{\cal H}}$, 
\end{itemize}
but it also inherits from the hierarchical B-spline basis $\hat{{\cal H}}$ the following properties:
\begin{itemize}
\item[(ii)] non-negativity: 
$\hat{\tau}\ge 0, \,\forall\,\hat{\tau}\in\hat{{\cal T}}$;
\item[(iii)] linear independence: 
$\sum_{\hat{\tau}\in\hat{{\cal T}}}c_{\hat{\tau}} \hat{\tau} = 0 \Leftrightarrow c_{\hat{\tau}} = 0$, $\forall\,\hat{\tau}\in\hat{{\cal T}}$;
\item[(iv)] nested nature of the {\Rd hierarchical} spline spaces {\Rd for consecutive levels;}
%:
%$\myspan{\hat{{\cal T}}^\ell\subseteq\myspan\hat{{\cal T}}^{\ell+1}}$;
\item[(v)] the span of a THB-spline basis defined over a sequence of subdomains is contained in the span of a truncated basis defined over a second sequence that is the nested enlargment of the original subdomain hierarchy;
%if one of the two sequences of nested subdomains $(\Omega^\ell)_{\ell=0,\ldots,N-1}$ and $(\hat{\Omega}_\ell)_{\ell=0,\ldots,\hat{N}-1}$, with $N\le\hat{N}$, is obtained by enlarging the other one while keeping the biggest subdomain, i.e., $\Omega^0=\hat{\Omega}^0$ and $\Omega^\ell\subseteq\hat{\Omega}^\ell$ for $\ell=1,\dots,N-1$, then $\myspan{{\cal H}}\subseteq\myspan{\hat{\cal H}}$;
\item[(vi)] completeness of the basis: for a certain class of admissible configurations of the hierarchical mesh, ${\myspan \hat{{\cal T}}}$ contains all piecewise polynomial functions defined over the underlying grid.
\end{itemize}
In addition, the truncation mechanism enriches the THB-spline basis functions so that
\begin{itemize}
\item[(vii)] they preserve the coefficients of the underlying sequence of B-splines;
\item[(iix)] they form a partition of {\Rd unity}; %on $\hat{\Omega}^0$: $\sum_{\hat{\tau}\in\hat{{\cal T}}}c_{\hat{\tau}} \hat{\tau} = 1$;
\item[(ix)] they are strongly stable with respect to the supremum norm, under reasonable assumptions on the given knot configuration.\footnote{Strong stability of a basis means that the associated stability constants do not depend on the number of hierarchical levels.}
\end{itemize}
Due the two-scale relation \eqref{eq:2scale} between adjacent (non-negative) B-spline bases, the non-negativity of truncated basis functions (ii) is preserved by construction. Properties (i), (iii)-(v), and (vii)-(ix) are detailed in \cite{giannelli2012,giannelli2014}. For the analysis of hierarchical spline space in (vi), we refer to \cite{giannelli2013} for the bivariate case with single knots, and to \cite{mokris2014a} for the general multivariate setting with arbitrary knot multiplicities.  As a consequence of property (vii), quasi-interpolants in hierarchical spline spaces can be easily constructed \cite{sm2015}. 
 Properties (ii) and (iix) imply the convex hull property, a key attribute for geometric modeling applications. 
%------------------------------------------------------------------
\subsection{Admissible meshes}
\label{sec:ameshes}
%------------------------------------------------------------------
%\textcolor{green}{Let $h_Q$ be the diameter of $Q$ and $h$ the maximum diameter of the elements $Q\in{\cal Q}$. We assume that ${\cal Q}$ is a shape regular and quasi uniform mesh, namely there exists a constant $C$ so that
%\[
%C h \le\diam(Q) \qquad \text{and} \qquad
%Ch\le |Q|^{1/d}, \qquad\forall Q\in{\cal Q}.
%\]}

The truncation mechanism that characterizes the THB-spline basis can be properly exploited to design suitable refinement strategies that define different \emph{classes of admissible meshes}. A mesh of this kind allows to guarantee that the number of basis functions acting on any mesh point is bounded. In addition, the support of any basis function acting on a single element of an admissible mesh can be compared with the size of the element itself in terms of two constants that do not depend on the {\Rd overall} number of hierarchical levels. These two properties are the key ingredients for the subsequent analysis --- see e.g., Theorem~\ref{eq:ub} related to the a posteriori upper bound, and the error indicator reduction provided by Lemma~\ref{lma:erred}. We postpone the presentation of the refinement procedure to Section~\ref{sec:mark&refine}, by simply focusing here on the desired mesh configuration.

\begin{dfn}\label{dfn:amesh} 
A mesh $\hat{{\cal Q}}$ is admissible of class $m$ if the truncated basis functions in $\hat{{\cal T}}(\hat{{\cal Q}})$ which take non-zero values over any element $\hat{Q}\in\hat{{\cal Q}}$ belong to at most $m$ {\Rd successive} levels.
\end{dfn}

For this class of admissible meshes, the number of basis functions acting on a single mesh element does not depend on the number of {\Rd levels in the hierarchy} but only on $m$, that represents the class of admissibility of the mesh.

\begin{crl}\label{crl:thbcor1} 
For multivariate tensor-product B-splines of degree {\Rd $\mathbf{p}=(p_1,$ $\ldots,$ $p_d)$}, the number of truncated basis functions which are non-zero on each element of an admissible mesh is then at most $m\prod_{i=1}^d(p_i+1)$.
\end{crl}

Another important fact that holds for admissible meshes is the following.

\begin{crl}\label{crl:thbcor2} If $\hat{\cQ}$ is an admissible mesh of class $m$, given a truncated basis function $\hat{\tau}\in \hat{{\cal T}}(\hat{{\cal Q}})$, 
\begin{equation}
  \label{eq:supp-basis}
 |\hat{Q}| \lesssim  | \supp \hat{\tau} | \lesssim |\hat{Q}| \qquad \forall \hat{Q}\in\hat{{\cal Q}}  \ :\ \hat{Q}\cap\supp \hat{\tau} \neq \emptyset,
\end{equation}
where the hidden constants in the above inequalities depend on $m$ but not on $\hat{\tau}$, neither on $\hat{\cQ}$ or $N$.
\end{crl}

In what follows, we will always indicate any inequality which does not depend on the depth $N$ of the spline hierarchy with $\lesssim$. 

Since the interplay between the truncated basis funcions that are non-zero on a certain mesh element and the overall mesh configuration is strictly related to the \emph{locality} of the basis functions, we naturally focus on the \emph{support extension} of an element $\hat{Q}\in\hat{\cG}^\ell$ .
For any fixed level $\ell$, the support extension collects the elements intersected by the set of B-splines in $\hat{{\cal B}}^\ell$ whose support overlaps $\hat{Q}$.
%namely
%\[
%\tilde{Q} = \bigcup\left\{
%Q'\in G^\ell: \supp\beta\cap Q'\ne\emptyset \wedge \supp\beta\cap Q\ne\emptyset,
%\beta\in\cB^\ell\right\},
%\]
We extend this definition to the hierarchical setting as follow. 
\begin{dfn}\label{dfn:hse}
The support extension $S(\hat{Q},k) $ of an element {\Rd $\hat{Q}\in\hat{G}^\ell$} with respect to level $k$, with $0\le k\le \ell$,
is defined as
\[
S(\hat{Q},k) {\Rd\,:=\,} \left\{
\hat{Q}'\in \hat{G}^k: 
{\Gd \exists\, \hat{\beta}\in \hat{\cB}^k,\,}
\supp\hat{\beta}\cap \hat{Q}'\ne\emptyset \wedge \supp\hat{\beta}\cap \hat{Q}\ne\emptyset
\right\}.
\]
\end{dfn}
{\Rd By a slight abuse of notation, we will also denote by $S(\hat{Q},k) $ the region occupied by the closure of elements in $S(\hat{Q},k)$.}
In order to identify a specific set of admissible meshes, we also consider the auxiliary subdomains
\[
{\Rd {\hat{\omega}}^{\ell} \,:=\,} \bigcup\left\{
\overline{\hat{Q}} {\Rd \,:\, \hat{Q}} 
\in \hat{G}^{\ell} {\Rd \,\wedge\,}
S(\hat{Q},{\ell})\subseteq \hat{\Omega}^{\ell}
\right\},
\]
%\[
%\Omega^{\ell+1}\subseteq\left\{
%\mathbf{x}\in\Omega^\ell | 
%\forall\beta\in{\cal B}^{\ell}:
%\mathbf{x}\in\supp\beta\Rightarrow
%\supp\beta\subseteq\Omega^\ell\right\},
%\]
for $\ell=0,\dots,N-1$. Any $\hat{\omega}^\ell$ represent the biggest subset of $\hat{\Omega}^\ell$ so that the set of B-splines in $\hat{\cal B}^\ell$ whose support is contained in $\hat{\Omega}^\ell$ spans the restriction of $V^\ell$ to $\hat{\omega}^\ell$.

\begin{exm}\label{exm:01}
A set of admissible meshes of class $m=2$ corresponds to the restricted hierarchies presented in Appendix A of \cite{giannelli2014} and relies on the following result. If $\hat{\Omega}^{\ell} \subseteq \hat{\omega}^{\ell-1}$ for $\ell=1,\ldots,N-1$, then for any element $\hat{Q}\in {\Rd \hat{{\cal G}}^{\ell}}$ the THB-splines whose support overlaps $\hat{Q}$ belong to at most two different levels: $\ell-1$ and $\ell$. Figure~\ref{fig:exm01} shows three examples of this class of admissible meshes related to the bivariate case of degree $(p_1,p_2)=(p,p)$ for $p=2,3,4$.

\begin{figure}[ht!]\begin{center}\hspace*{-.75cm}
%\subfigure[$(p_1,p_2)=(1,1)$ ]{
%\includegraphics[scale=0.195]{exm02p1a}}\hspace*{-.75cm}
\subfigure[$(p_1,p_2)=(2,2)$ ]{
\includegraphics[scale=0.25]{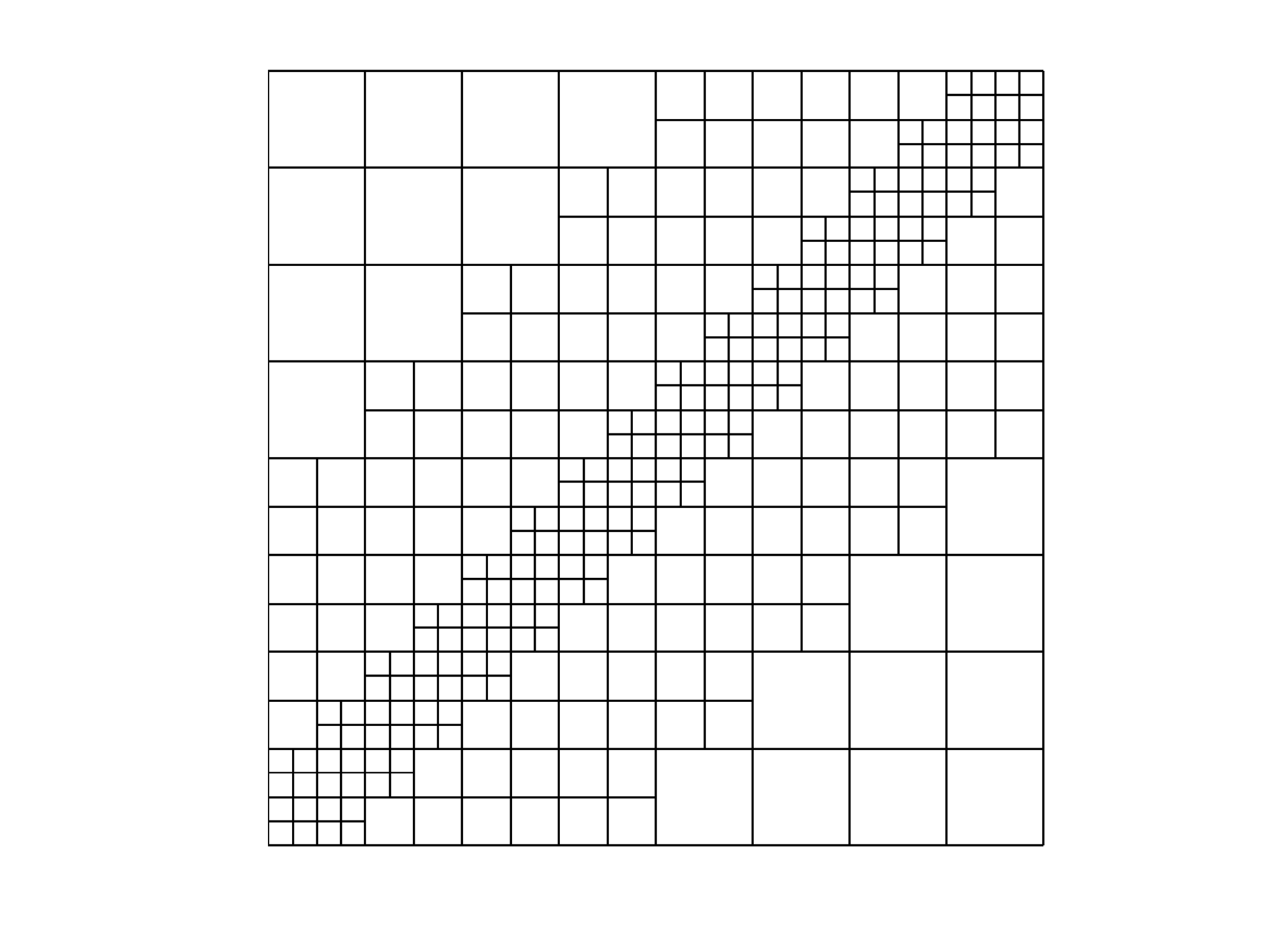}}\hspace*{-.75cm}
\subfigure[$(p_1,p_2)=(3,3)$ ]{
\includegraphics[scale=0.25]{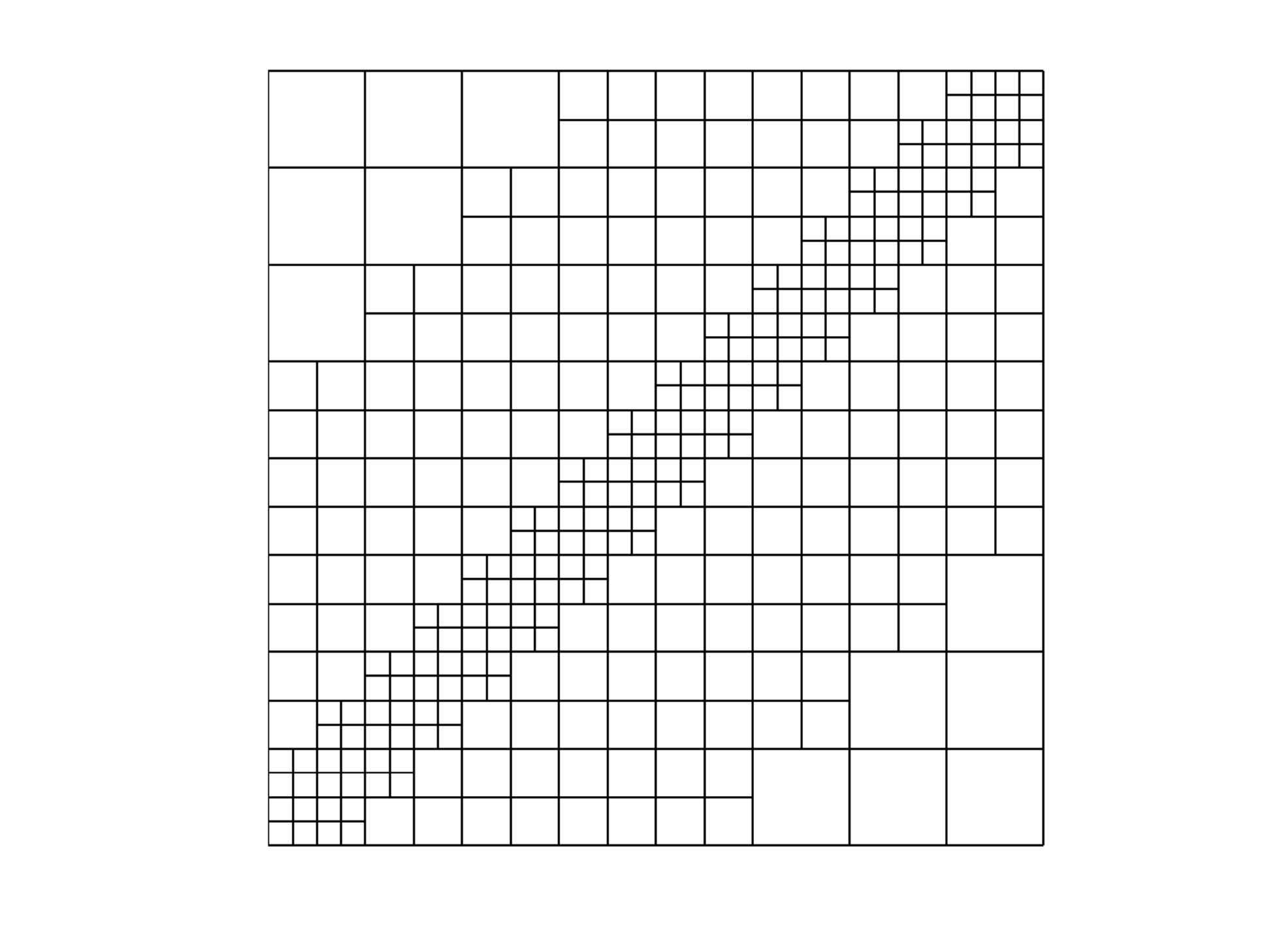}}\hspace*{-.75cm}
\subfigure[$(p_1,p_2)=(4,4)$ ]{
\includegraphics[scale=0.25]{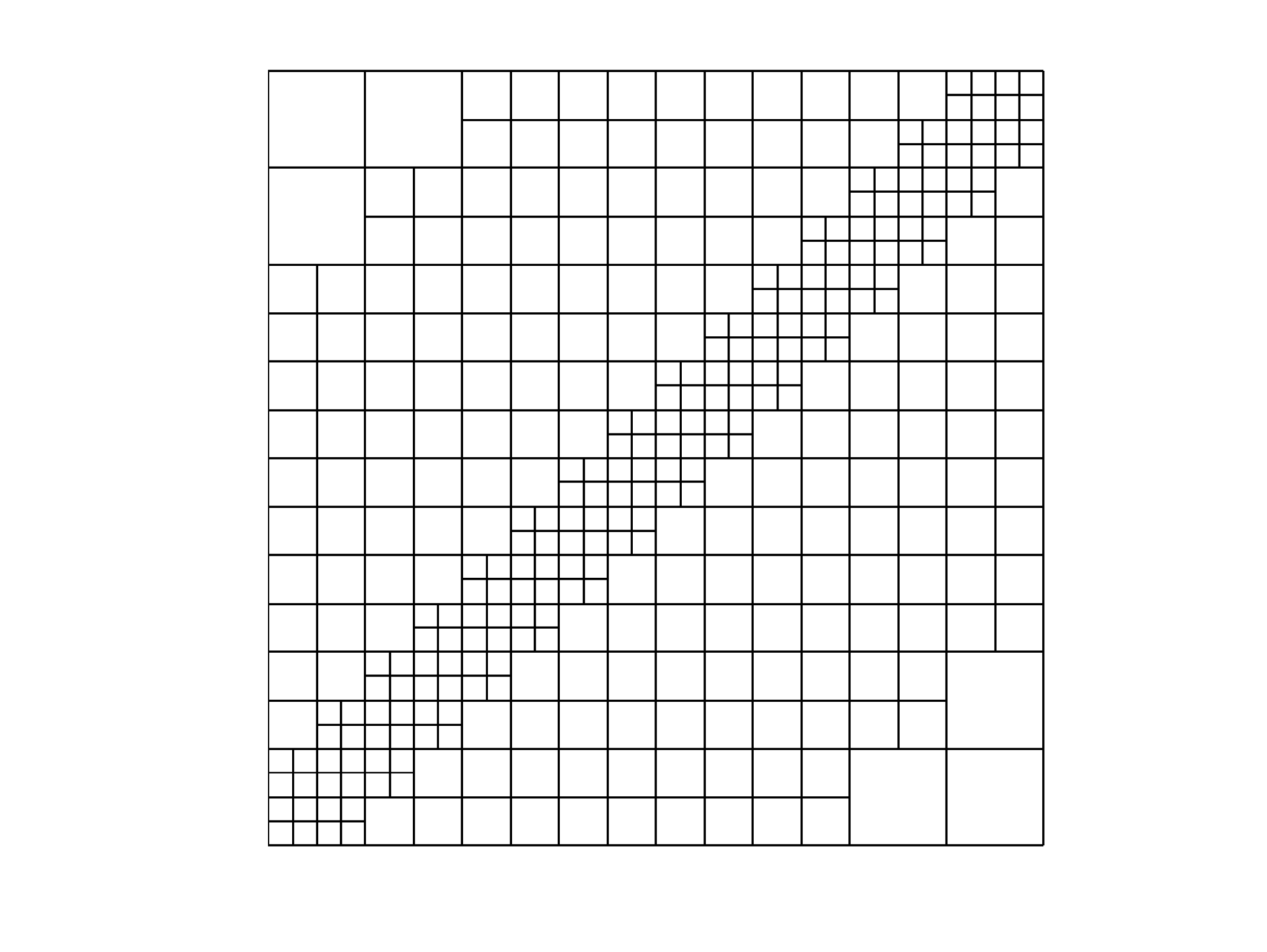}}
\caption{Admissible meshes of class 2 considered in Example~\ref{exm:01} with respect to different degrees.}
\label{fig:exm01}
\end{center}\end{figure}

%\begin{figure}[ht!]\begin{center}\hspace*{-1.5cm}
%\subfigure[$(p_1,p_2)=(1,1)$ ]{
%\includegraphics[scale=0.25]{exm02p1b}}\hspace*{-1.5cm}
%\subfigure[$(p_1,p_2)=(2,2)$ ]{
%\includegraphics[scale=0.25]{exm02p2b}}\hspace*{-1.5cm}
%\subfigure[$(p_1,p_2)=(3,3)$ ]{
%\includegraphics[scale=0.25]{exm02p3b}}\hspace*{-1.5cm}
%\subfigure[$(p_1,p_2)=(4,4)$ ]{
%\includegraphics[scale=0.25]{exm02p4b}}
%\caption{...examples of admissible meshes...}
%\label{fig:exm02b}
%\end{center}\end{figure}

%\begin{figure}[ht!]\begin{center}\hspace*{-.75cm}
%%\subfigure[$(p_1,p_2)=(2,2)$ ]{
%%\includegraphics[scale=0.25]{exm02p2a}}\hspace*{-.75cm}
%\caption{Admissible meshes of class $3$ considered in Example~\ref{exm:01} with respect to different degrees.}
%\label{fig:exm02b}
%\end{center}\end{figure}

\end{exm}

The following proposition generalizes the class of admissible meshes considered in the previous example to the case of an arbitrary $m\ge 2$.
%, a set of admissible meshes of class $m$ can be identified by
%the following condition:
\begin{prn}\label{prop:adm}
Let $\hat{Q}$ be the mesh of active elements defined according to \eqref{eq:active} and \eqref{eq:mesh} with respect to the domain hierarchy $\hat{\Omega}^0\supseteq\hat{\Omega}^1\supseteq\ldots\supseteq\hat{\Omega}^{N-1}$. If
\begin{equation}\label{eq:sameshes}
\hat{\Omega}^\ell\subseteq \hat{\omega}^{\ell-m+1},
\end{equation}
for $\ell=m,m+1,\ldots,N-1$, then the mesh $\hat{{\cal Q}}$ is admissible of class $m$.
\end{prn}
\begin{proof}
For any $\hat{\tau}\in\hat{{\cal T}}(\hat{\cal Q})$ introduced at level $\ell-m$, the function $\trunc^{\ell-m+1}\hat{\tau}$ defined by equation \eqref{eq:trunc} is a linear combination of basis functions $\hat{\beta}\in\hat{{\cal B}}^{\ell-m+1}$ so that $\hat{\beta}\vert_{\hat{\omega}^{\ell-m+1}}=0$. Since
\[
\hat{\omega}^{\ell-m+1} = {\Rd \bigcup} 
\left\{
{\Rd \overline{\hat{Q}}:\,}
\hat{Q}\in \hat{G}^{\ell-m+1} {\Rd \,\wedge\,} 
S(\hat{Q},{\ell-m+1})\subseteq
\hat{\Omega}^{\ell-m+1}\right\},
\] 
if condition \eqref{eq:sameshes} holds for $\ell=m,m+1,\ldots,N-1$, then also $\trunc^{\ell-m+1}\hat{\tau}\vert_{\hat{\Omega}^\ell}=0$. Consequently, the truncation of a B-spline introduced at level $\ell-m$ will be non-zero on $\hat{\Omega}^{\ell-m}\setminus\hat{\Omega}^{\ell}$. This means that any element $\hat{Q}\in\hat{{\cal G}}^\ell$ belongs to the support of THB-splines of only $m$ different levels: {\Rd $\ell-m+1, \ldots, \ell$}.
\end{proof}

As we will detail later, a relevant set of admissible meshes is the one verifying condition \eqref{eq:sameshes} for $\ell=m,m+1,\ldots,N-1$, where different values of $m\ge2$ can be considered.
\begin{dfn}\label{dfn:samesh} 
A mesh $\hat{{\cal Q}}$ is strictly admissible of class $m$ if it verifies the assumptions of Proposition~\ref{prop:adm}. 
\end{dfn}
The meshes considered in Example~\ref{exm:01} are strictly admissible of class $2$.
%------------------------------------------------------------------
\section{The module SOLVE:  the Galerkin method}
\label{sec:solve}
%\marginlabel{\Gd why not simply "model problem"?\B}
In this section we describe our model problem  and introduce its discretization by means of hierarchical splines.   Indeed,  we have no aim of generality, we consider the most simple elliptic problem.
As a first step, we give a precise definition of the domain $\Omega$ in which our problem is posed. 
%%% def of Omega %%%%

Given a strictly admissible mesh $\hat{\cal Q}_0$ and the corresponding set of truncated basis function $\hat{\cal T}_0$,  we suppose that the computational domain $\Omega$ is provided as a linear combination of functions in $\hat{\cal T}_0$ and control points: 

\begin{equation}
  \label{eq:1}
  \bx\in {\Rd\overline{\Omega}}\,, \quad  \bx = \bF(\hat{\bx}) = \sum_{\hat\tau\in \hat{{\cal T}}_0} \mathbf{C}_{\hat{\tau} } \hat{\tau}(\hat{\bx})\qquad \hat{\bx} \in 
  {\Rd \hat{\Omega}^0}
\end{equation}
where $\mathbf{C}_{\hat{\tau}}\in \mathbb{R}^d$. 
In all what follows, we suppose that the mapping {\Rd $\bF: \hat{\Omega}^0\to \overline{\Omega}$} is a bi-Lipschitz homeomorphism:
\begin{equation}
  \label{eq:Fbound}
 \| D^\alpha\bF\|_{L^{\infty}({\Rd \hat{\Omega}^0})} \le C_{\bF}, 
\quad 
\| D^\alpha\bF^{-1}\B \|_{L^{\infty}( \Omega)}   
	\le c_{\bF}^{-1},  \qquad |\alpha|\leq 1
\end{equation}
where $c_\bF$ and  and  $C_\bF$   are independent constants bounded away from infinity.
%% end def of Omega%%%%

We consider then the following problem:

\begin{equation}
\label{eq:mp_1}
-\div (\bA \nabla u) = f \quad\text{in}\; \Omega, 
\qquad 
u\myvert{\partial\Omega} = 0,
\end{equation} 
where $\bA\in C^\infty (\bar{\Omega})$ is the diffusion matrix which verifies \eqref{eq:5}.

In order to define the variational formulation of the problem, 
we consider the space of functions in $H^1(\Omega)$ with vanishing trace on $\partial\Omega$
\[
\mathbb{V}:=H_0^1(\Omega):=\left\{
v\in H^1(\Omega):v\myvert{\partial\Omega}=0
\right\},
%\qquad
%||v||_{\mathbb{V}}:=
\]
endowed with the norm $\| u \|^2_{\VV} = \| \nabla v\|^2_{L^2(\Omega)^d} + \|  v \|^2_{L^2(\Omega)}$. 
A weak solution of \eqref{eq:mp} is a function $u\in \VV$ satisfying
\begin{equation}\label{eq:weak}
u\in\mathbb{V}:\quad
a(u,v) = \langle f,v \rangle, \quad \forall\, v\in\mathbb{V},
\end{equation}
where $a:\mathbb{V}\times\mathbb{V}\rightarrow\mathbb{R}$ is the bilinear form 
%\begin{align*}
%&a:\mathbb{V}\times\mathbb{V}\rightarrow\mathbb{R} 
%\text{ is the bilinear form } 
\[
a(u,v):=\int_{\Omega} \mathbf{A} \nabla u \nabla v, \qquad
\forall\, u,v\in\mathbb{V},
\]
%&
%\langle \cdot,\cdot \rangle
%:=\int_\Omega  \nabla \cdot \nabla \cdot 
%\text{ stands for the $L^2(\Omega)$ scalar product}.
%\end{align*}
and $\langle \cdot,\cdot \rangle$ stands for the $L^2(\Omega)$ scalar product. We assume that $f\in\mathbb{V}^*$. 
The bilinear form $a(u,v)$ is coercive and continuos with constant $\alpha_1$ and $\alpha_2$, respectively:
\begin{align}
  \label{eq:coer}
  & a(u,u)  \geq  \alpha_1 \| u\|^2_{\VV} \qquad & u \in {\Rd \VV ,}\\
\label{eq:cont}
 & a(u,v)  \leq  \alpha_2 \| u \|_\VV \|v \|_\VV\qquad & u\,,\ v \in {\Rd \VV .}
\end{align}
Moreover, it induces the \emph{energy norm}: 
$|||v|||_\Omega:=a(v,v)^{1/2}$, $\forall v\in\mathbb{V}$. 
The coercivity and continuity properties of $a(u,v)$ implies the equivalence between the energy and the $H^1(\Omega)$ norms on $\mathbb{V}$. In addition, the Lax-Milgram theorem ensures the existence and uniqueness of the weak solution \eqref{eq:weak}.

\smallskip

   {We construct now our module SOLVE as the Galerkin discretization of \eqref{eq:weak} by means of hierarchical splines on $\Omega$. To this aim, we first need to introduce a suitable notation for hierarchical meshes and spaces  on $\Omega$. }

We consider an admissible mesh $\hat{\cal Q}$, such that  $\hat{{\cal Q}} \succeq \hat{{\cal Q}}_0$ and we denote by $\hat{{\cal T}}$ the corresponding basis truncated basis functions. Moreover, we construct the corresponding mesh and functions of the physical domain via pullback:
\[ {\cal Q} = \{  Q =  \bF(\hat{Q}): \B \hat{Q}\in \hat{{\cal Q}}\}.\] 
For all $\hat \tau\in \hat{{\cal T}}$, we construct:
\begin{equation}
  \label{eq:4}
  \tau (\bx) = \hat{\tau}(\hat{\bx}), \qquad \bx = \bF(\hat{\bx}). 
\end{equation}
    and we denote by ${\cal T}$ the collection of all mapped basis functions, and by $\SS(\cQ)$ the space they generate, $ \SS(\cQ) =\myspan{\cal T} ({\cal Q})$. 

%\marginlabel{\Gd ma come e' definita qui $|Q|$?\B}
Clearly,  ${\cal Q}$ is a hierarchical mesh on the domain $\Omega$ and for it, we will make use of all the nomenclature introduced in Section \ref{sec:hspaces}  by simply \B removing the $\hat{\cdot}$. First, for all elements $Q$, we denote by $\hat{Q}$ its preimage through $\bF$, i.e., $Q= \bF(\hat{Q})$, and $h_Q = |Q|^{1/d}$, where $|Q|$ represents the volume of $Q$.   Moreover, we set: 
\begin{itemize}
\item $\Omega^\ell = \bF(\hat\Omega^\ell)$ and $\omega^\ell = \bF(\hat\omega^\ell)$;
\item ${\cal G}^\ell =\{Q \in {\cal Q} \ :\ \hat{Q}\in \hat{\cal G}^\ell\}$ and $G^\ell = \{ Q\subset \Omega \ :\ \hat{Q}\in \hat{G}^\ell\}$;
\item for all ${Q}\in {\cal G}^\ell$,  its support extension with respect to  level $k$ is 
\[S(Q, k) = \{ Q'\in G^k \ :\ \hat{Q}'\in S(\hat{Q},k)\}.\]
\end{itemize}
Finally,  when ${\cal Q}^\star$ is a refinement of ${\cal Q}$, we will write ${\cal Q}^\star \succeq {\cal Q}$, when their pre-images $\hat{{\cal Q}}^\star$ and $\hat{{\cal Q}}$ verifies $\hat{{\cal Q}}^\star \succeq \hat{{\cal Q}}$. 

%\marginlabel{\Gd ho solo unito i paragrafi qui \B}
We are finally in the position to describe the discrete problem we want to solve adaptively. The Galerkin approximation of \eqref{eq:weak} consists in solving:
\begin{equation}\label{eq:gal}
\text{find } \ U\in\mathbb{S}_D(\cQ): \quad
a(U,V) = \langle  f,V \rangle,\quad \forall\, V\in\mathbb{S}_D(\cQ),
\end{equation}
where 
\[
\SS_D(\cQ) = \left\{V\in \SS({\cal Q}): V\myvert{\partial\Omega} = 0\right\}.
\]
In the subsequent analysis we assume for simplicity $\SS_D({\cal Q})\subset C^1(\Omega)$. This assumption is of course not needed for the development of an adaptive strategy, but it allows us to simplify the analysis, by also showing the specific changes with respect to $C^0$ finite elements. The general case could be treated in a similar way following the classical theory of adaptive finite element methods.  
%------------------------------------------------------------------
\section{The module ESTIMATE: the residual based error indicator}
\label{sec:estimate}
%------------------------------------------------------------------
The residual associated to $U\in\mathbb{S}$ is the functional in $\mathbb{V}^*$ defined by
\begin{equation*}
  \langle R,v\rangle  := \langle f,v\rangle -a(U,v), %\quad \forall v\in\mathbb{V},
\end{equation*}
that satisfies
\begin{align*}
&\langle r,v\rangle = a(u-U,v), \quad\forall\, v\in\mathbb{V},\\
&a(u-U,V) = \langle r,V\rangle = 0, \quad \forall\, V\in\mathbb{S}. 
\end{align*}
By recalling that all discrete functions are continuous with continuous derivatives, we can 
integrate by parts and  obtain
\begin{equation*}
\langle r,v\rangle = \int_{\Omega} fv - \bA \nabla U\nabla v 
= \int_\Omega fv - \div (\bA \nabla U)  v 
,
\end{equation*}
where, thanks to our assumption that $\mathbb{S} \subset C^1(\Omega)$, the quantity $r= f - \div (\bA \nabla U) $ belongs to $L^2(\Omega)$. In particular, as we expect, this means that the residual does not contain any edge contribution as in typical finite element indicators
\cite{verfurth2013}.

One of the fundamental ingredient in the module ESTIMATE is  the equivalence between the primal norm of the error and the dual norm of the residual:
\begin{equation}\label{eq:ee1}
||u-U||_{\mathbb{V}} \le 
\frac{1}{\alpha_1} ||r||_{\mathbb{V}^*}\le 
\frac{\alpha_2}{\alpha_1}||u-U||_{\mathbb{V}}.
\end{equation}
As it is standard, in order to use the residual as error indicator, we would like to replace the norm $\|\cdot\|_{\mathbb{V}*}$ with the following error indicator
\begin{equation}\label{eq:ind}
\varepsilon^2_{\cal Q}(U, {\cal Q}) 
= \sum_{Q\in{\cal Q}} \varepsilon_{\cal Q}^2(U,Q) 
%\end{equation}
\qquad\text{with}\qquad
%\begin{equation*}
\varepsilon^2_{\cal Q}(U,Q)
= h_Q^2 ||r||_{L^2(Q)}^2. 
\end{equation}
When no confusion is possible, we may also abbreviate the above notation with $\varepsilon^2_{\cal Q}(U)$. 

Following \cite{morin2001} (see also \cite{NochettoCIME}), we will show that the following holds: 

\begin{equation}
\label{eq:apost-1}
||u-U||_{\mathbb{V}} \lesssim
% \frac{1}{\alpha_1}
\varepsilon_{\cal Q}(U,{\cal Q}) \lesssim
% \frac{\alpha_2}{\alpha_1}
||u-U||_{\mathbb{V}}
+% \frac{1}{\alpha_1}
\osc_{\cal Q}(U, {\cal Q}),
\end{equation}
where 
\begin{equation*}
  \label{eq:oscillaz}
  \osc^2_{\cal Q}(U, {\cal Q}) = \sum_{Q\in {\cal Q}} \osc^2(U,Q) \quad \text{with} \quad \osc(U,Q) = h_Q \| r-\Pi_\bn r\|_{L^2(Q)}
\end{equation*}
and $\Pi_\bn : L^2(Q) \to \mathbb{Q}_\bn$, $\bn=(n_1,n_2,n_3)$,  denotes the $L^2$ projector onto the space of polynomials of degree $n_j$ in the space direction $j$. The degrees $n_j$, $j=1,\ldots,d$ can be fixed large enough so that the oscillation are ``smaller'' than the error
\cite{bonito2010}. 

Indeed Theorem \ref{thm:lb} below, will also provide a local version of the lower bound in \eqref{eq:apost-1} that reads:
\begin{equation*}
  \label{eq:lb-local-1}
  \varepsilon_{\cal Q}(U,{Q}) \lesssim
% \frac{\alpha_2}{\alpha_1}
||u-U||_{\mathbb{V}(Q)}
+% \frac{1}{\alpha_1}
\osc_{\cal Q}(U, {Q}).
\end{equation*}

%------------------------------------------------------------------
\subsection{A posteriori upper bound}%\marginlabel{$\Rightarrow$ used for contraction}
%------------------------------------------------------------------
In this section we prove that the residual based error indicator defined in \eqref{eq:ind} is \emph{reliable}, i.e., it is an upper bound for the Galerkin error.

\begin{thm}\label{thm:ub}
Let $u$ be the exact weak solution of the model problem \eqref{eq:weak}. The error of the Galerkin approximation $U\in\mathbb{S}({\cal Q})$ in \eqref{eq:gal} is bounded in terms of the error indicator $\varepsilon_{\cal Q}(U)$ introduced in \eqref{eq:ind} as follows:
\begin{equation}
\label{eq:ub}
||u-U||_{\mathbb{V}} \leq C_{\mathrm{up}} \varepsilon_{\cal Q}(U),
\end{equation}
where the constant $C_{\mathrm{up}}$ is independent on the mesh size and on the level of hierarchy.
\end{thm}
\begin{proof}
This proof follows exactly the lines of the classical proof of upper bound in residual based error estimators. For completeness we repeat here the steps that can be found in, e.g., Theorem 6 in \cite{NochettoCIME}.

Using (\ref{eq:ee1}), we have $\| u-U\|_{\VV} \lesssim \displaystyle \frac{1}{\alpha_1}  \| r \|_{\VV^\star}$, and we will prove that $ \| r \|_{\VV^\star} \les \varepsilon_\cQ(U)$.  

Since the basis functions in ${\cal T}$ form a partition of unity and the residual is orthogonal to all basis functions $\tau$ in ${\cal T}$, it holds:
\[
\langle r, v\rangle  = \sum_{\tau \in {\cal T}}\langle r, \tau \, v \rangle = \sum_{\tau\in {\cal T}}  \inf_{c_\tau \in \mathbb{R} } \langle r, \tau  \, (v-c_\tau) \rangle.
 \]

By standard Cauchy-Schwarz inequality, we estimate the terms in the right hand side as follows:
 
\[\langle R, \tau\, (v-c_\tau) \rangle = \int_\Omega 
r\, \tau  (v-c_\tau) \leq \| r\, \tau^{1/2} \|_{L^2(\Omega)}  \| \tau^{1/2} (v-c_\tau) \|_{L^2(\Omega)}.   \]
We denote by $\omega_\tau =\supp\tau$ and  by $h_{\omega_\tau} = |\supp \tau|^{1/d}$, i.e., its size.    We can deduce by Poincar\'e inequality that:
\[ \| \tau^{1/2} (v-c_\tau )\|_{L^2(\omega_\tau)}  \les  h_{\omega_\tau} \| \nabla v\|_{L^2(\omega_\tau)^d}.\]

By taking into account Corollaries \ref{crl:thbcor1}  and \ref{crl:thbcor2}, we have
\begin{enumerate}
\item $\sum_{\tau\in {\cal T}} \| \nabla v\|^2_{L^2(\omega_\tau)^d} \les \| \nabla v\|^2_{L^2(\Omega)^d}  $; % where $C_p$ depending on $\mathbf{p}$ and is given in   Corollary \ref{cor:bound}, $C_p= 2\prod_{i=1}^d(p_i+1)$;
\item let $h$ be the piecewise constant function which takes values $h(\bx) = |Q|^{1/d}$, $\bx\in Q$ for all $Q\in \cQ$. It holds: 
\[  \sum_{\tau\in {\cal T}} h^2_{\omega_\tau}   \| r\, \tau^{1/2} \|^2_{L^2(\omega_\tau)}  \les  \sum_{\tau\in {\cal T}}  \int_{\omega_\tau} h^2\,   r^2 \tau  = \int_\Omega h^2\,   r^2 = \varepsilon_\cQ(U). \]
\end{enumerate}

The estimate (\ref{eq:ub}) follows. 
%\[  \langle R, v\rangle   \les \sum_{i=1}^{n_\cQ}  \| r \tau_i^{1/2} \|_{L^2(\omega_i)} h_{\omega_i} \| \nabla v \|_{L^2(\omega_i)^d}   \]

\end{proof}

%------------------------------------------------------------------
\subsection{A posteriori lower bound}%\marginlabel{$\Rightarrow$ used for convergence}
%------------------------------------------------------------------

 In this section we prove that the residual based error indicator defined in \eqref{eq:ind} is \emph{efficient}, i.e., it is a lower bound of the Galerkin error up to oscillations. \B

\begin{thm}
  \label{thm:lb}
Let $u$ be the exact weak solution of the model problem \eqref{eq:weak}. The error of the Galerkin approximation $U\in\mathbb{S}({\cal Q})$ in \eqref{eq:gal} bounds the  error indicator $\varepsilon_{\cal Q}(U)$ introduced in \eqref{eq:ind} up to oscillations:

\begin{equation}
  \label{eq:lb-local-2}
  \varepsilon_{\cal Q}(U,{Q}) \leq C_{\mathrm{lb}} \big(
% \frac{\alpha_2}{\alpha_1}
||u-U||_{\mathbb{V}(Q)}
+% \frac{1}{\alpha_1}
\osc_{\cal Q}(U, {Q})\big),
\end{equation}
where the  constant $C_{\mathrm{lb}} $ does not depend on $Q$. 
\end{thm}
\begin{proof}
  Again, this proof is classical, and we repeat the steps of the proof in Theorem 7 of \cite{NochettoCIME}.
First, it is easy to see that 
\[ \| r\|_{\mathbb{V}^*(Q)} \les \| \nabla (u-U)\|_{L^2(Q) }\]
and that the following Poincar\'e estimate is true: 
\[ \| r \| _{\mathbb{V}^*(Q)} \les h_Q \| r\|_{L^2(Q)}. \]
Moreover, let $\mathbb{Q}_\bn$, $\bn=(n_1,n_2,n_3)$ be the space of polynomials of degree $n_j$ in the space direction $j$, then we know that the inverse inequality holds: 
\[ \| \bar{r} \|_{\mathbb{V}^*} \gtrsim h_Q \| \bar{r} \|_{L^2(Q)} \qquad \forall \bar{r} \in  \mathbb{Q}_\bn,\]
where the hidden constant does not depend on $Q$ but it deteriorates with $\bn$. 

Finally, if we choose $\bar{r} = \Pi_\bn r$, it holds that $\| \bar{r}\|_{\mathbb{V}^*(Q)}  \leq \| r  \|_{\mathbb{V}^*(Q)} $. 

Now, all these ingredients can be used together in the following estimate, with $\bar{r} = \Pi_\bn r$:
\begin{equation}
  \label{eq:2}
  \begin{aligned}
    h_Q \|r\|_{L^2(Q)} & \leq h_Q \| \bar{r} \|_{L^2(Q)} + \|r-\bar{r}\|_{L^2(Q)} \\
& \les \|\bar{r}\|_{\mathbb{V}^*(Q)}  + h_Q \| r-\bar{r}\|_{L^2(Q)} \\
& \les  \|{r}\|_{\mathbb{V}^*(Q)}  + h_Q \| r-\bar{r}\|_{L^2(Q)} 
  \end{aligned}
\end{equation}
The proof is completed by setting $\osc_{\cal Q} (U,Q) = h_Q \| r-\bar{r}\|_{L^2(Q)}$. 
\end{proof}

\begin{rmk}
   It should be noted that the lower bound we have proved here will not be used in the sequel of the present paper. In fact, contraction of the error can be proved without using explicitly the lower bound. We have reported this simple proof here in order to collect the main properties of the estimator we are using. On the other hand, this lower bound will be needed in the companion paper \cite{buffa2015b} where optimality will be  addressed. 
\end{rmk}
%------------------------------------------------------------------
\section{The modules MARK and REFINE}
\label{sec:mark&refine}
%------------------------------------------------------------------
 We now briefly describe the considered marking strategy, before introducing a refine module that preserves the class of admissibility of a given strictly admissible mesh --- see Section~\ref{sec:ameshes}
%\ref{sec:hspaces} 
--- and its properties. Finally, we conclude this section by discussing the contraction property of our AIGM and its convergence. \B
%------------------------------------------------------------------
\subsection{MARK: the marking strategy}
%\label{sec:mark&refine}
%------------------------------------------------------------------
Given an admissible mesh ${\cal Q}$, the Galerkin solution $U\in\mathbb{V({\cal Q})}$, the module 
\[
{\cal M} = \text{MARK}\left(
\left\{\varepsilon_{\cal Q}(U,Q)\right\}_{Q\in{\cal Q}},
{\cal Q}\right),
\]
selects and marks a set of elements ${\cal M}\subset{\cal Q}$ according to the so-called \emph{D\"orfler marking} \cite{dorfler1996}, i.e., by considering a fixed parameter $\theta\in(0,1]$ so that 
\begin{equation}\label{eq:dm}
\varepsilon_{\cal Q}(U,{\cal M}) \ge \theta\,
\varepsilon_{\cal Q}(U,{\cal Q}).
\end{equation}
This marking strategy simply guarantees that the set ${\cal M}$ of marked elements gives a substantial contribution to the total error indicator.
%------------------------------------------------------------------
\subsection{REFINE: the refinement strategy}
%\label{sec:refine}
%------------------------------------------------------------------
The support extension $S(\widehat{Q},k)$ of an element $\widehat{Q}\in\widehat{\cG}^\ell$ with respect to level $k$, with $0\le k\le \ell$
introduced in Definition~\ref{dfn:hse};
\Gd analogously, we denote by $S(Q,k)$ the support extension of the physical element $Q$. 

In order to guarantee that a mesh after refinement is  admissible, we aim at imposing that each active element   at level $\ell$, $Q\in \cG^\ell$, belongs to the support of basis functions of at most levels $\ell-m+1$, \ldots, $\ell$. To achieve this, given an element $Q\in \cG^\ell$ that  we want to refine, we select the elements that are active, that are of level $\ell-m+1$ and such that their intersection with  $S(Q,\ell-m+2)$ is not empty. We collect all these elements in a neighborhood of $Q$, denoted by ${\cal N}({\Rd {\cQ},\,}Q,m) $ and  defined here below in Definition \ref{dfn:neigh}. Clearly, when $Q$ is refined, all elements in  ${\cal N}({\Rd {\cQ},\,}Q,m) $ have to be refined as well and  this procedure has to be applied recursively in order to guarantee that the final mesh is strictly admissible.  \B 

 % allows us to select the elements of level $k$ 
% intersected by the set of B-splines in ${\cal B}^k$ that have non-zero value on $\widehat{Q}$. {Analogously, $S(Q,k)$ will be the support extension of an element ${Q}\in{\cG}^\ell$.} In order to exploit
% the enhanced locality of the truncated basis, we aim to iteratively ``cover'' every marked  element $Q\in{\cal M}\cap\cG^\ell$ 
% only with active B-splines of levels $k, k+1, \ldots, \ell-1, \ell$,   for some $k\le \ell$. To achieve this, the elements in the
% \emph{neighborood} of $Q\in\cG^\ell$ with respect to $\ell$ should be of level greater or equal to $\ell-m+2$ for
% an admissible mesh of class $m$. This guarantees that, when the marked element $Q$ will be splitted in the refined elements of level $\ell+1$, on every of these elements only basis functions of (at most) level $\ell-m+2,\ell-m+3,\ldots,
% \ell,\ell+1$ may have non-zero value.
\begin{dfn}\label{dfn:neigh}
The neighborood of $Q\in {\Rd {\cQ} \,\cap\,}\cG^\ell$ with respect to $m$ is defined as
\[
{\cal N}({\Rd {\cQ},\,}Q,m) {\Rd \,:=\,} \left\{Q'\in{\cal G}^{\ell-m+1}: 
{\Rd \exists\, Q'' }\in S(Q,\ell-m+2) 
{\Rd , Q''\subseteq Q' } \right\},
\]
when ${\Rd \ell-m+1 > 0}$, and ${\cal N}({\Rd {\cQ},\,} Q,m) = \emptyset$ for ${\Rd \ell-m+1 \le 0}$.
\end{dfn}

 Figure~\ref{fig:exm02} shows the the neighborood of an element $Q$ with respect to $m=2$ when, for simplicity, the identity map is considered. 

\begin{figure}[ht!]\begin{center}
%\hspace*{-.75cm}
%\subfigure[$(p_1,p_2)=(1,1)$ ]{
%\includegraphics[scale=0.195]{exm01p1a}\hspace*{-.75cm}
%\includegraphics[scale=0.195]{exm01p1b}\hspace*{-.75cm}
%\includegraphics[scale=0.195]{exm01p1c}\hspace*{-.75cm}
%\includegraphics[scale=0.195]{exm01p1d}}\\
\hspace*{-.75cm}
\subfigure[$(p_1,p_2)=(2,2)$ ]{
\includegraphics[scale=0.195]{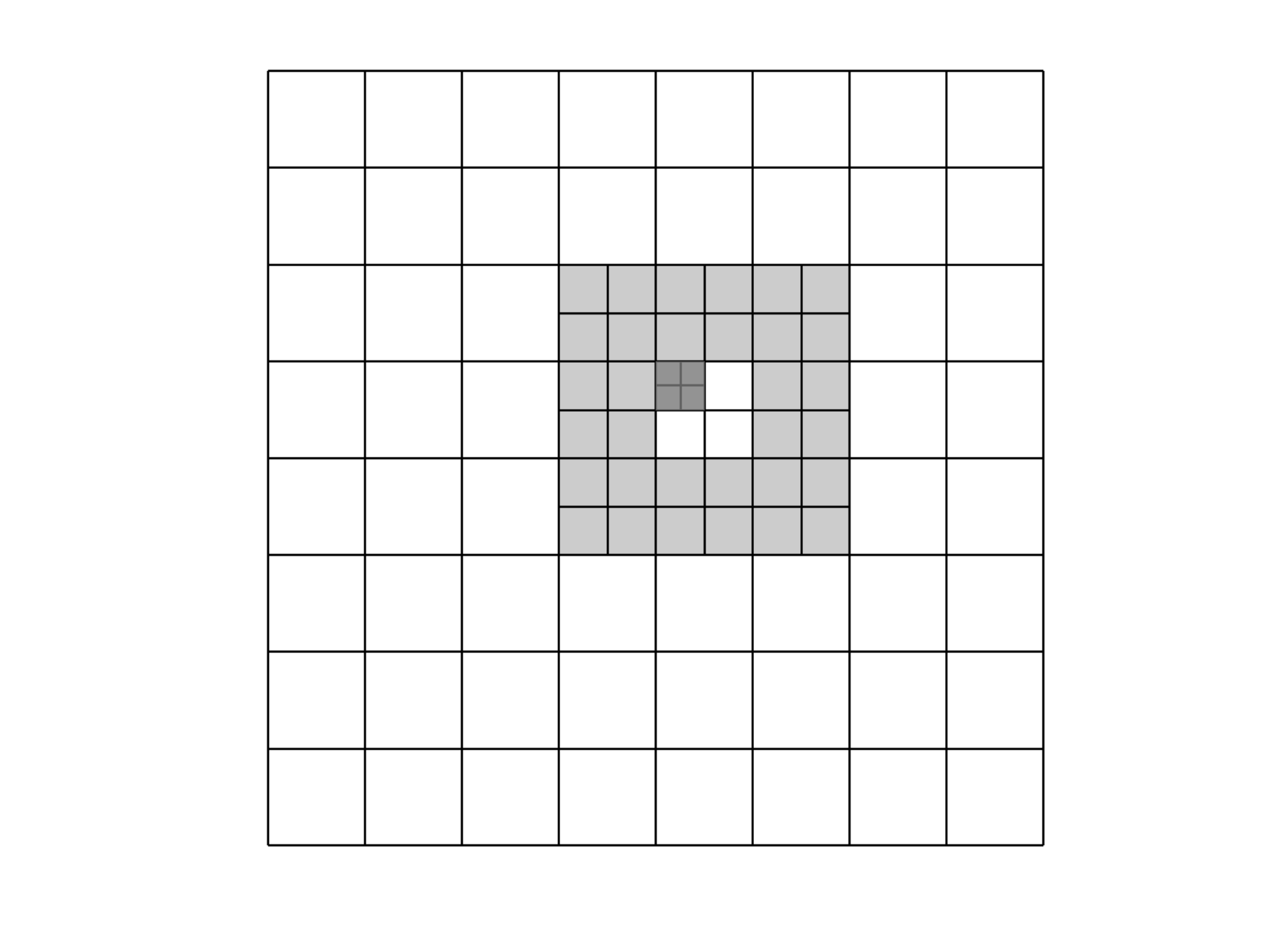}\hspace*{-.75cm}
\includegraphics[scale=0.195]{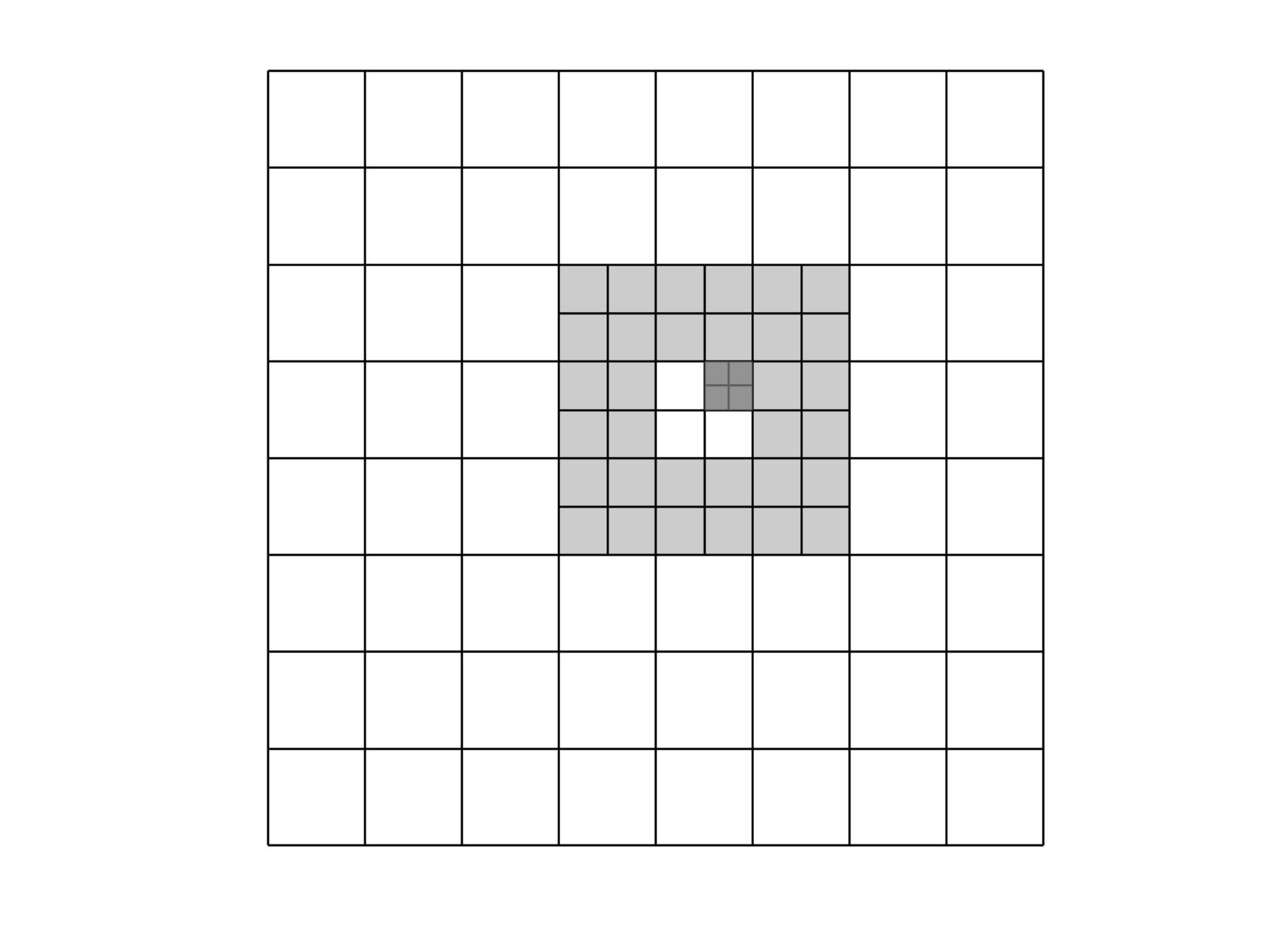}\hspace*{-.75cm}
\includegraphics[scale=0.195]{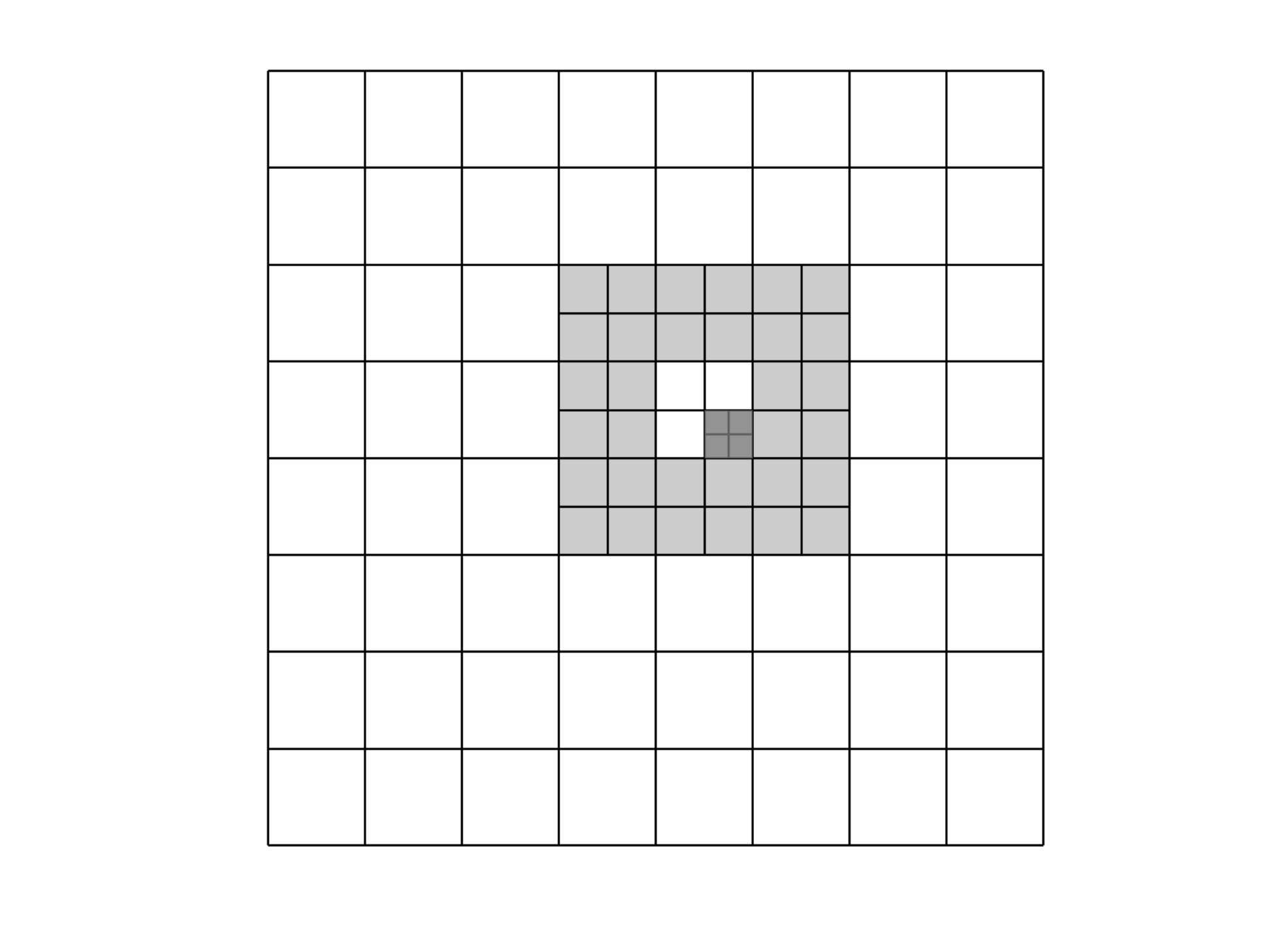}\hspace*{-.75cm}
\includegraphics[scale=0.195]{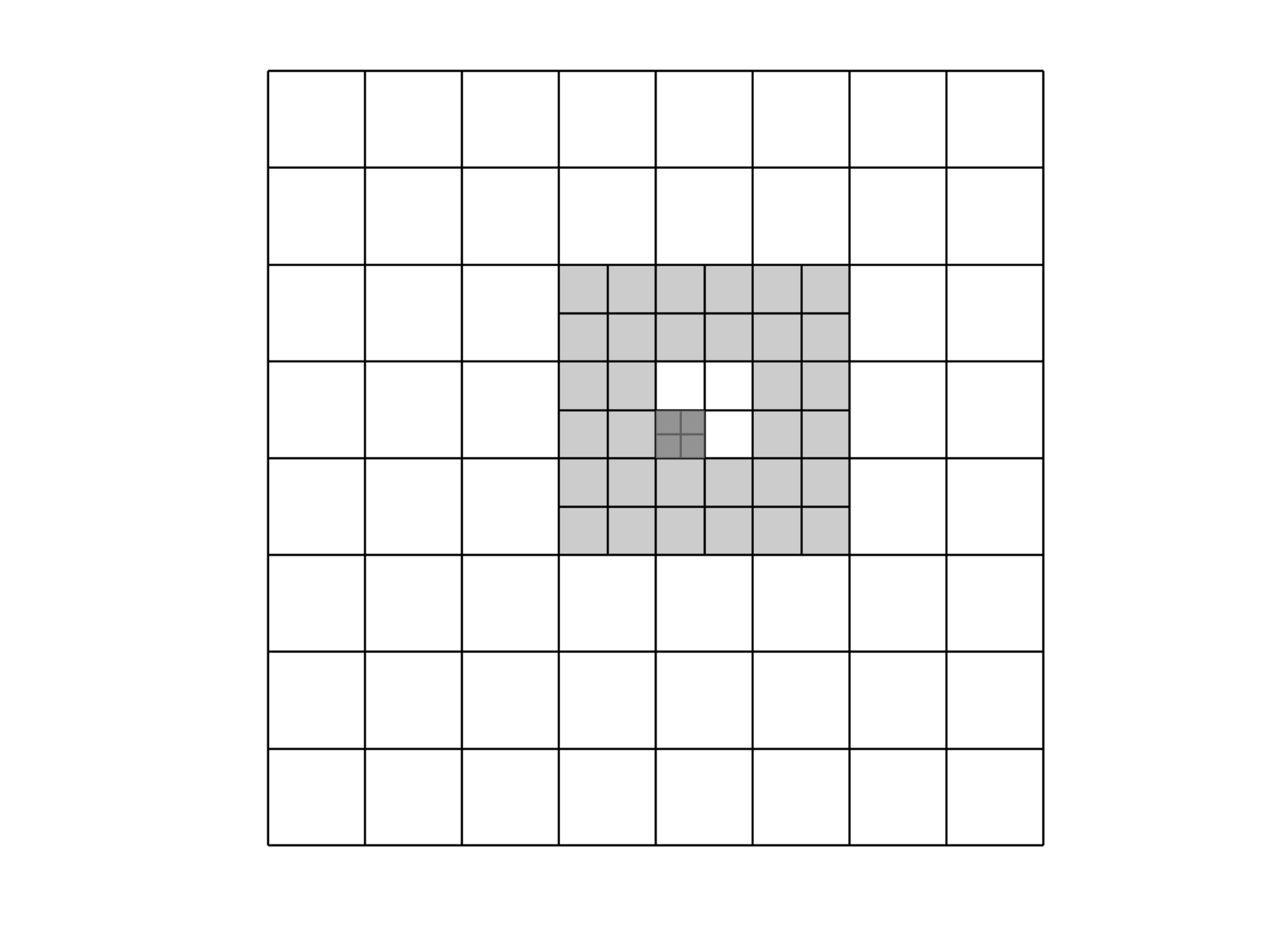}}\\
\hspace*{-.75cm}
\subfigure[$(p_1,p_2)=(3,3)$ ]{
\includegraphics[scale=0.195]{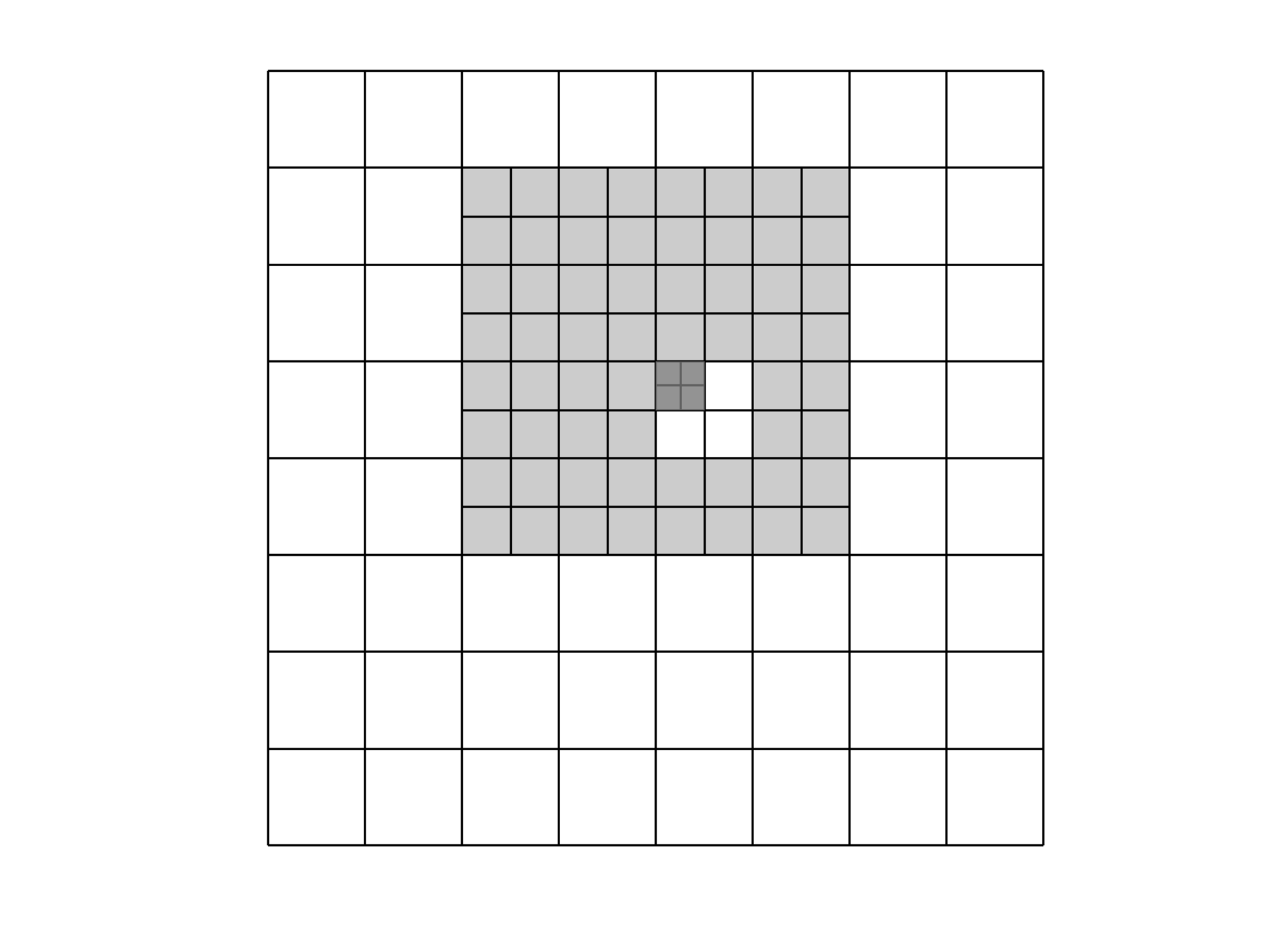}\hspace*{-.75cm}
\includegraphics[scale=0.195]{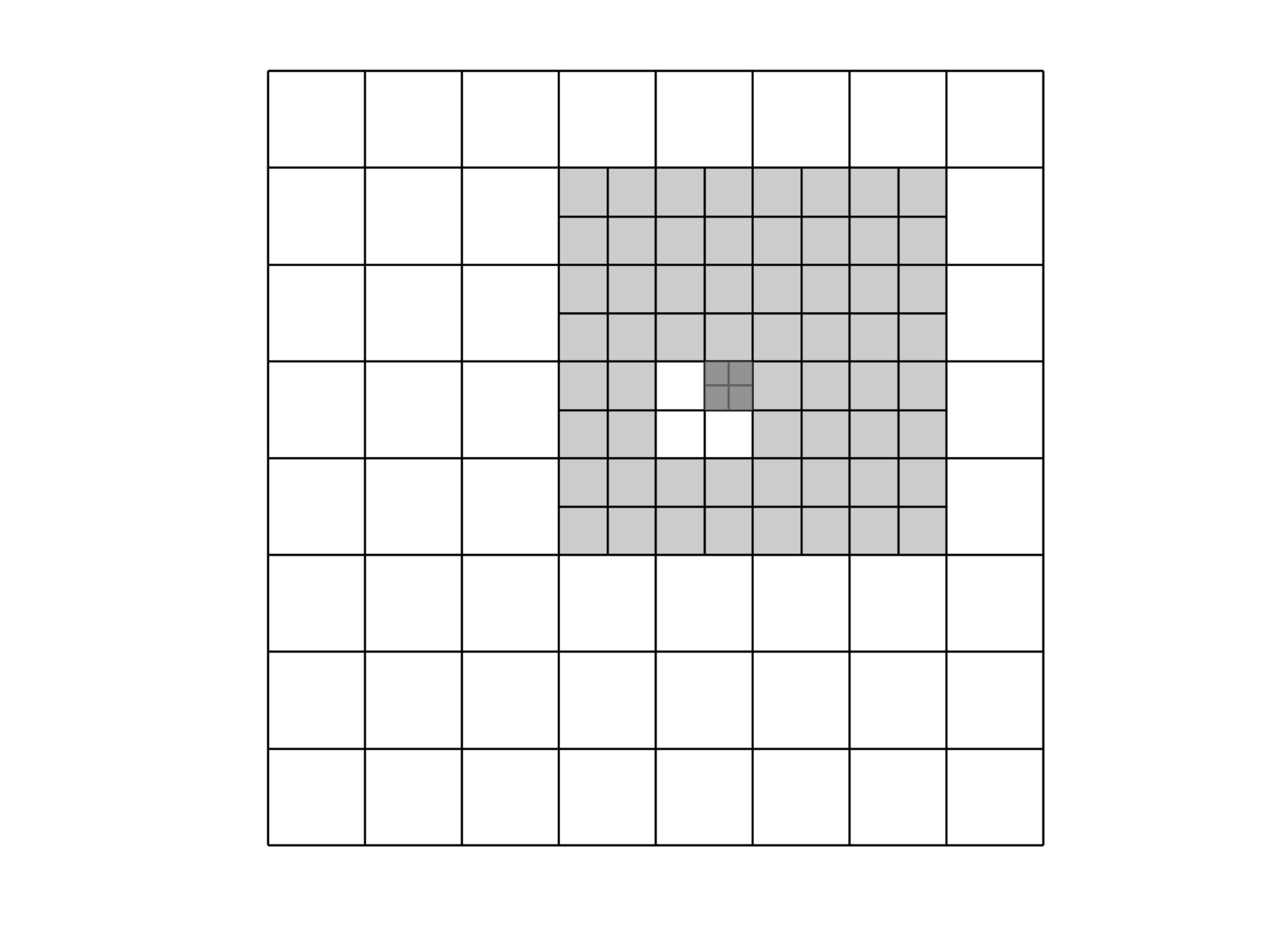}\hspace*{-.75cm}
\includegraphics[scale=0.195]{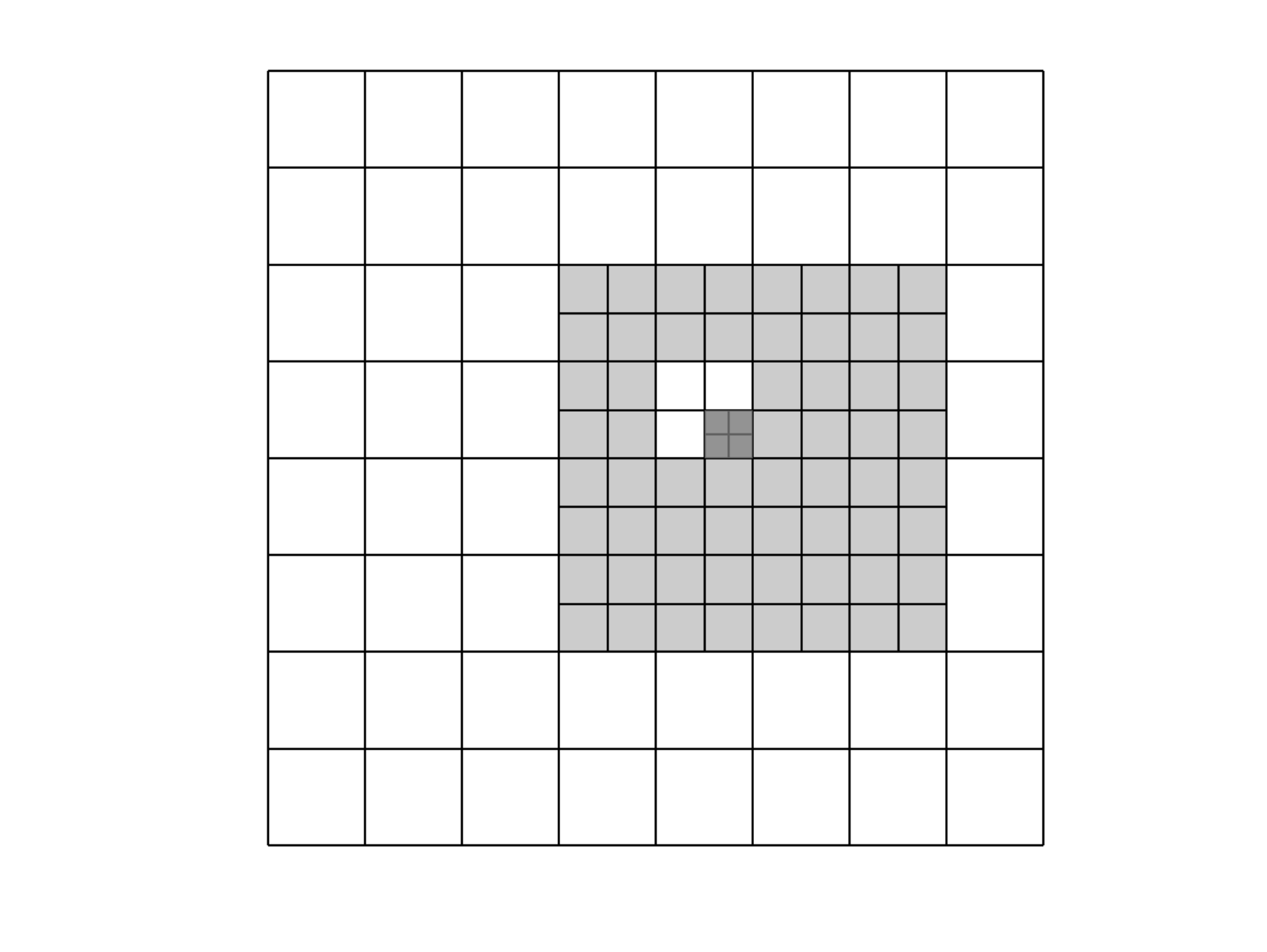}\hspace*{-.75cm}
\includegraphics[scale=0.195]{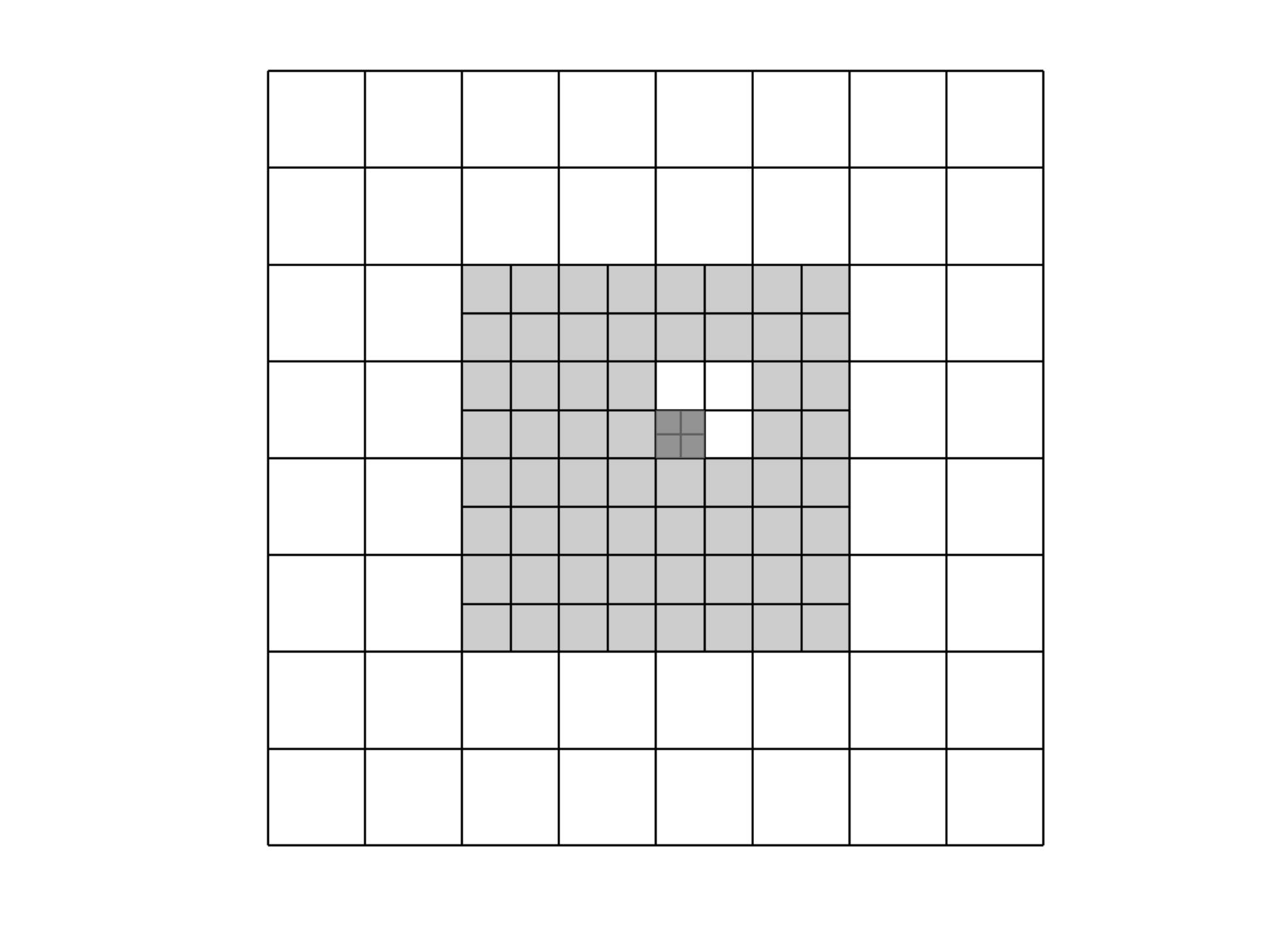}}\\
\hspace*{-.75cm}
\subfigure[$(p_1,p_2)=(4,4)$ ]{
\includegraphics[scale=0.195]{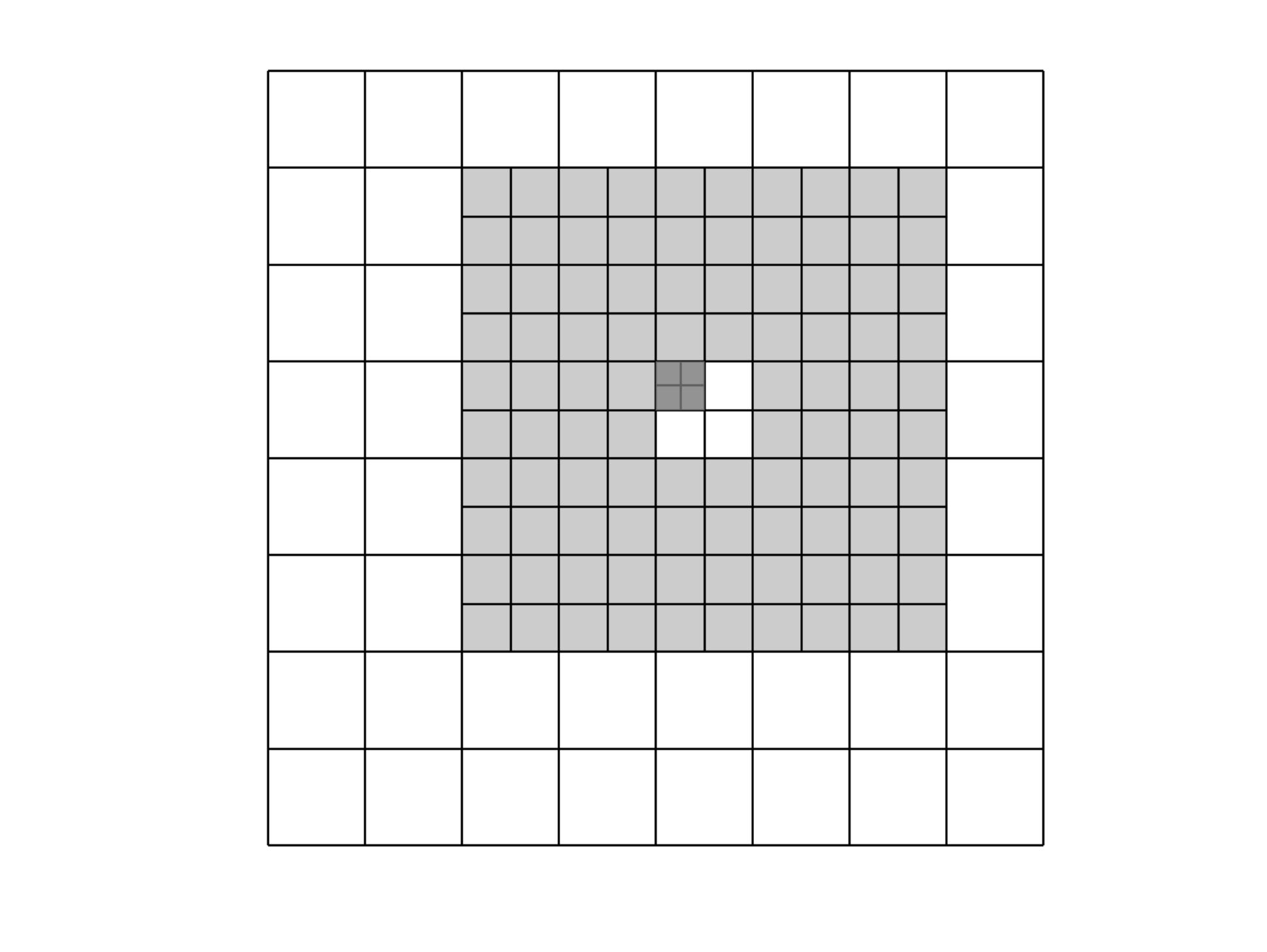}\hspace*{-.75cm}
\includegraphics[scale=0.195]{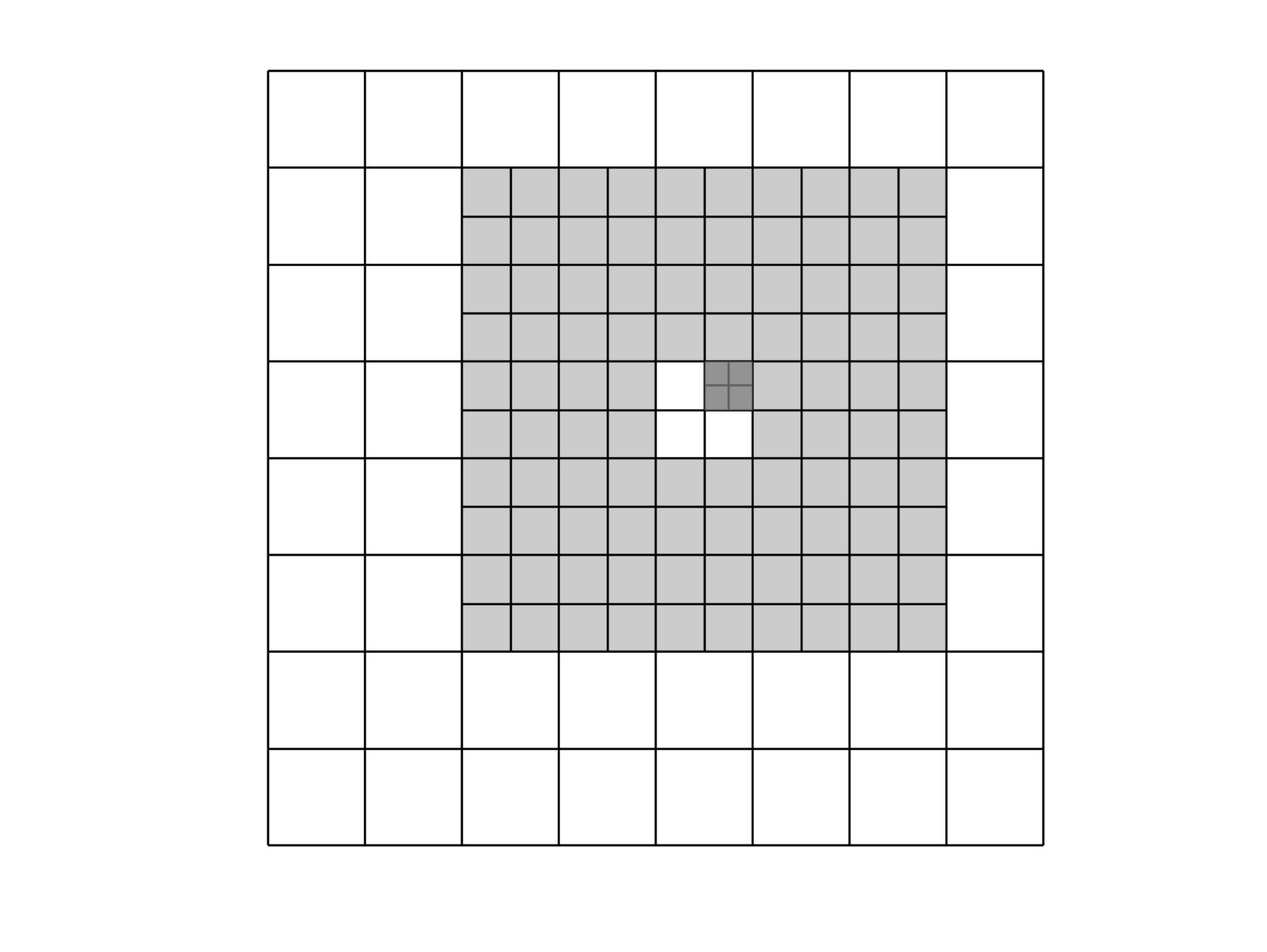}\hspace*{-.75cm}
\includegraphics[scale=0.195]{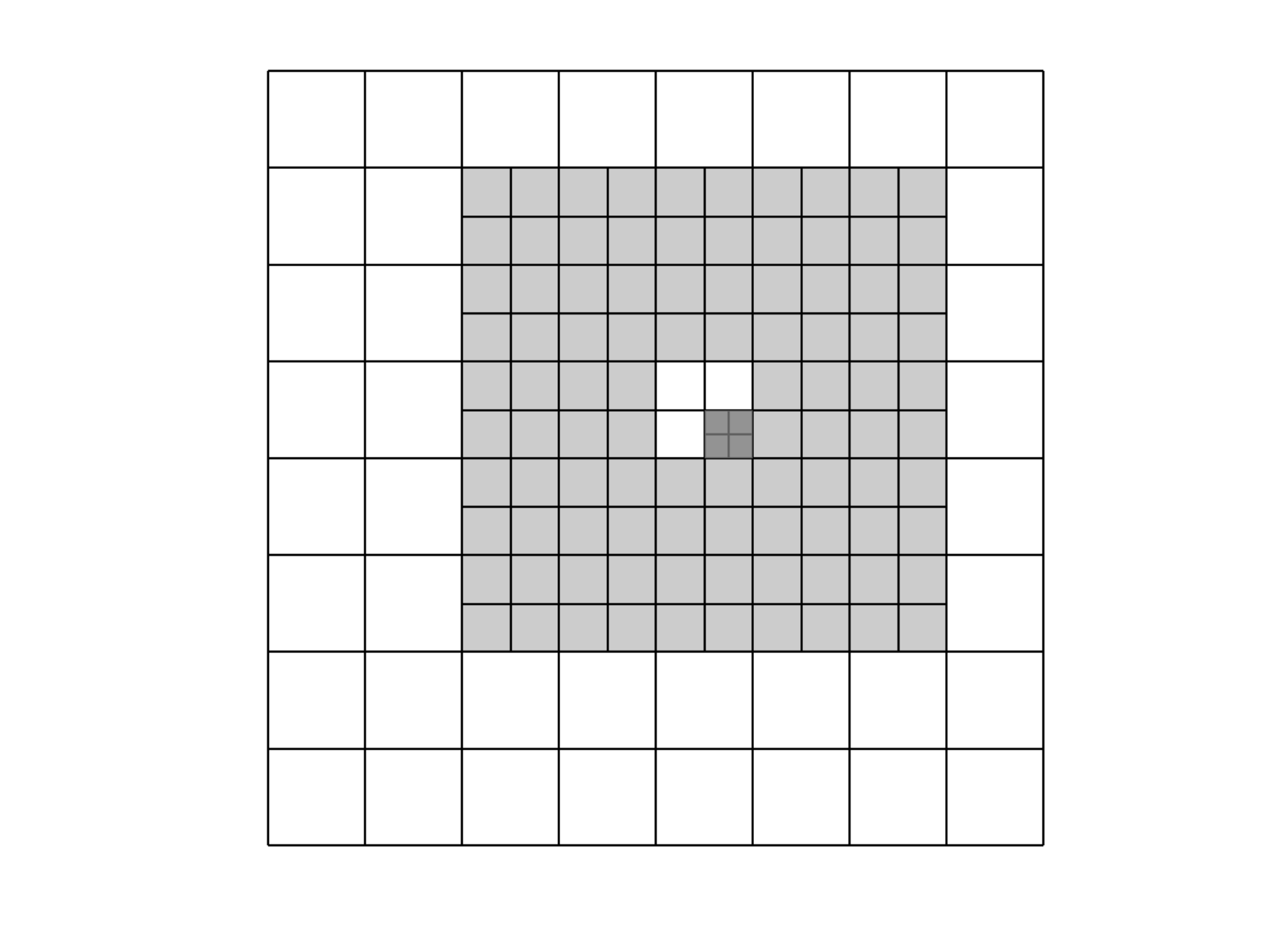}\hspace*{-.75cm}
\includegraphics[scale=0.195]{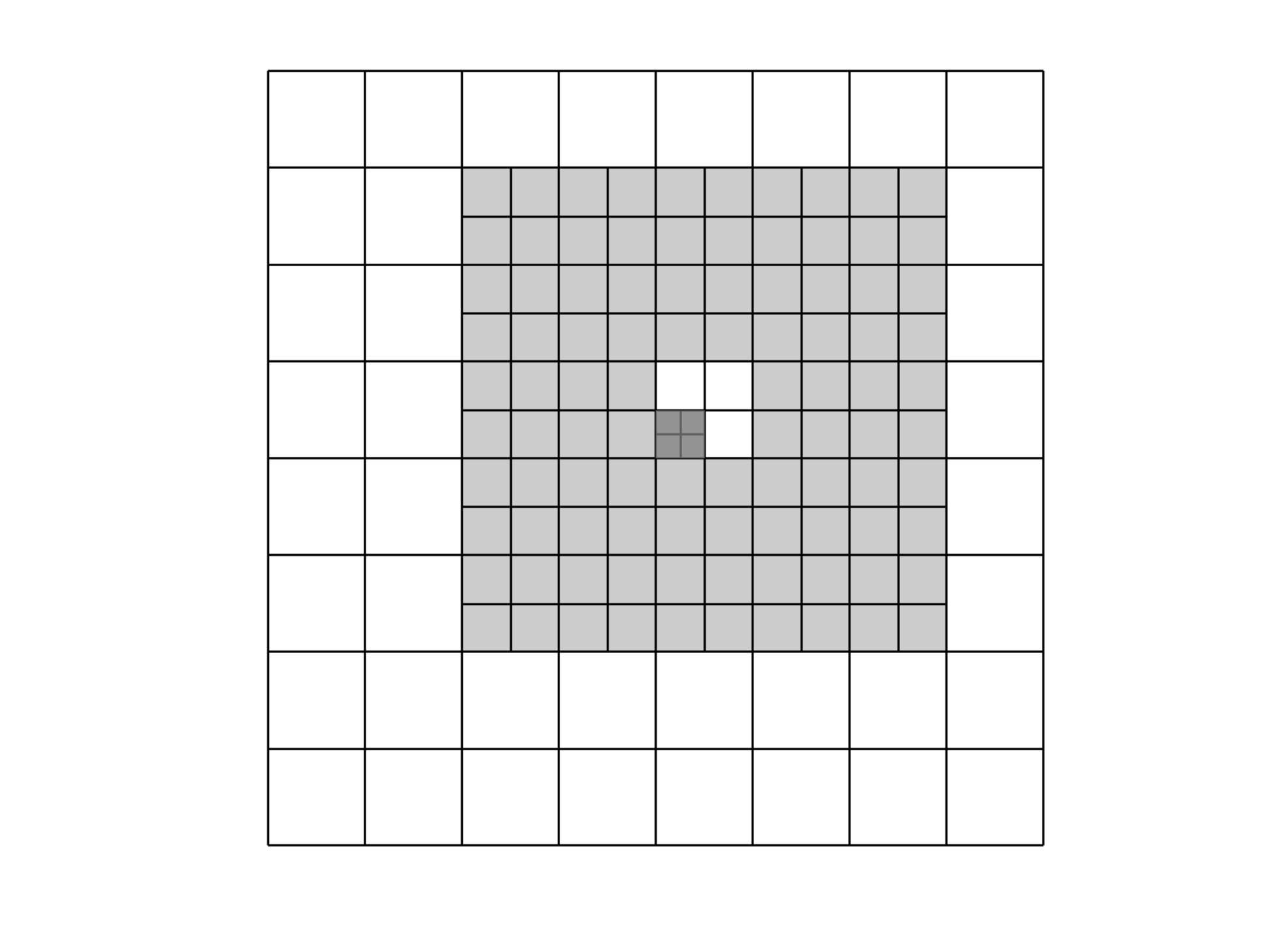}}
\caption{Neighborood ${\cal N}({\Rd \cQ,\,}Q,2)$ (light gray) of an element $Q$ ({represented by any of the four cells in dark grey}) when dyadic refinement is considered for some low degree cases. Note that the neighborood is always aligned with the grid lines of a previous hierarchical level.}
\label{fig:exm02}
 \end{center}\end{figure}

An automatic REFINE module which allows to define \emph{strictly} admissible meshes is presented in Figure~\ref{fig:refine}. The core of the refinement strategy relies on the internal recursive module REFINE$\_$RECURSIVE. {Any element $Q$ on which the recursive procedure is called will be subdivided into its children.} Lemma~\ref{lma:rr} and Proposition~\ref{prn:rr} below shows the distinguishing properties of this procedure.

% \begin{figure}[ht!]\begin{center}\begin{footnotesize}
% \begin{tabular}{l}
% %--------------------------------------------------------------------------------------------
% \hline \vspace*{-.35cm}\\
% ${\cal Q} = $ REFINE(\Rd${\cal Q}_0\B,{\cal M},m$)\\%\medskip\\
% \hline \vspace*{-.25cm}\\
% %--------------------------------------------------------------------------------------------
% \Rd $\cQ = \cQ_0$\B \smallskip\\
% for all $Q\in{\cal Q}\cap{\cal M}$\smallskip\\
% \hspace*{.5cm} ${\cal Q} = $ REFINE$\_$RECURSIVE(${\cal Q},Q,m$)\smallskip\\
% end\\
% \hline
% \end{tabular}\hspace*{.5cm}
% \begin{tabular}{l}
% %--------------------------------------------------------------------------------------------
% \hline \vspace*{-.35cm}\\
% ${\cal Q} = $ REFINE$\_$RECURSIVE($\cQ,Q,m$)\\%\medskip\\
% \hline \vspace*{-.25cm}\\
% %--------------------------------------------------------------------------------------------
% for all $Q_{\cal M} \in{\cal N}(Q,m)$\smallskip\\
% %of level $\ge 0$
% %\hspace*{.25cm} so that level($Q'$) = level($Q$) - 1\smallskip\\
% \hspace*{.5cm}${\cal Q} = $ REFINE$\_$RECURSIVE(${\cal Q},Q_{\cal M},m$)\smallskip\\
% end\smallskip\\
% if \Rd$Q\in\cQ_0$\B \Gd (i.e., $\cQ $ has not been subdivided)\smallskip\B \\
% \hspace*{.25cm} subdivide $Q$\smallskip\\
% \hspace*{.25cm} \Rd update $\cQ$ by replacing $Q$ with its children\smallskip\\
% end\\
% \hline
% \end{tabular}

\begin{figure}[ht!]\begin{center}\begin{footnotesize}
\begin{tabular}{l}
%--------------------------------------------------------------------------------------------
\hline \vspace*{-.35cm}\\
${\cal Q}^\star = $ REFINE(${\cal Q},{\cal M},m$)\\%\medskip\\
\hline \vspace*{-.25cm}\\
%--------------------------------------------------------------------------------------------
for all $Q\in{\cal Q}\cap{\cal M}$\smallskip\\
\hspace*{.5cm} ${\cal Q} = $ REFINE$\_$RECURSIVE(${\cal Q},Q,m$)\smallskip\\
end\\
$\cQ^\star = \cQ$ \smallskip\\
\hline
\end{tabular}\hspace*{.5cm}
\begin{tabular}{l}
%--------------------------------------------------------------------------------------------
\hline \vspace*{-.35cm}\\
${\cal Q} = $ REFINE$\_$RECURSIVE($\cQ,Q,m$)\\%\medskip\\
\hline \vspace*{-.25cm}\\
%--------------------------------------------------------------------------------------------
for all $Q' \in{\cal N}({\Rd \cQ,\,}Q,m)$\smallskip\\
%of level $\ge 0$
%\hspace*{.25cm} so that level($Q'$) = level($Q$) - 1\smallskip\\
\hspace*{.5cm}${\cal Q} = $ REFINE$\_$RECURSIVE(${\cal Q},Q',m$)\smallskip\\
end\smallskip\\
if  {\Rd$Q$} has not been subdivided\smallskip\B \\
\hspace*{.25cm} subdivide {\Rd$Q$} and \smallskip\\
\hspace*{.25cm} update $\cQ$ by replacing {\Rd$Q$} with its children\smallskip\\
end\\
\hline
\end{tabular}

\end{footnotesize}\end{center}
\caption{The REFINE and REFINE$\_$RECURSIVE modules.}\label{fig:refine}
\end{figure}

\begin{lma}\label{lma:rr}
(Recursive refinement) 
Let ${\cal Q}$ be a {strictly} admissible mesh of class $m$. The call to $\cQ^* =$ REFINE$\_$RECURSIVE(${\cal Q},Q,m$) terminates and returns a refined mesh $\cQ^*$ with elements that either were already active in ${\cal Q}$ or are obtained by single refinement of an element of $\cQ$.
\end{lma}
\begin{proof}
For every marked element $Q\in\cG^\ell\cap\cM$ the REFINE$\_$RECURSIVE routine is recursively called on any element of level $\ell'={\ell-m+1}$ that belongs to the neighborood of $Q$ with respect to $m$ while $\ell'$ is greater or equal than zero. Since at each recursive call the level $\ell'$ of interest  is strictly decreasing, the termination condition will be satisfied after a finite number of steps.
In addition, any element $Q$ touched by a call to REFINE$\_$RECURSIVE is subdived in its children only the first time it is reached in the return phase after the set of recursive calls. Every element of ${\cal Q}$ is then refined at most once in the refinement process that generates ${\cal Q}^*$ from ${\cal Q}$.
\end{proof}

By exploiting the truncation mechanism in the context of strictly admissible meshes --- see Definition~\ref{dfn:samesh} ---
it is possible to show that only the supports of truncated basis functions of level $\ell-m+1,\ell-m+2,\ldots,\ell$ will contain an element $Q\in{\cal G}^\ell$, for every refined mesh generated by the REFINE$\_$RECURSIVE module.

\begin{prn}\label{prn:rr} 
Let $\cQ$ be a strictly admissible mesh of class $m\ge 2$ and let $Q_\cM$ be an active element of level $\ell$, for some $0\le\ell\le N-1$. The call to $\cQ^* =$ REFINE$\_$RECURSIVE $(\cQ,Q_\cM,m)$ returns a strictly admissible mesh $\cQ^*\succeq\cQ$ of class $m$. 
\end{prn}
\begin{proof}
Let $\Omega^{0}\supseteq\ldots\supseteq\Omega^{N-1}\supseteq\Omega^{N}$, with $\Omega^N=\emptyset$, be the domain hierarchy associated to mesh ${\cal Q}$.
The refined mesh $\cQ^*$ = REFINE$\_$RECURSIVE $(\cQ,Q_\cM,m)$ contains active elements $Q^*\in\cG^{\ell,*}$ with respect to the domain hierarchy $\Omega^{0,*}\supseteq\ldots\supseteq\Omega^{N-1,*}\supseteq\Omega^{N,*}$ where 
\begin{equation}\label{eq:nested}
\Omega^{0,*}\equiv\Omega^{0} 
\quad\text{and}\quad 
\Omega^{\ell,*}\supseteq\Omega^{\ell},
\end{equation} 
for $\ell=1,\ldots,N$. Note that the maximum level of refinement in $\cQ^*$ is necessarily $N$ according to Lemma~\ref{lma:rr}.\\
Let $Q^*\in\cG^{\ell,*}$ be an active element of $\cQ^*$, then $Q^*\subseteq\Omega^{\ell,*}\setminus\Omega^{\ell+1,*}$, for some $0\le\ell\le N$. We have two possibilities: either $Q^*$ belongs also to $\Omega^\ell$ or not.
\begin{itemize}
\item If $Q^*\subseteq\Omega^\ell$ then $0\le\ell\le N-1$. Since the initial mesh $\cQ$ is strictly admissible of class $m$, we have: $\Omega^\ell\subseteq\omega^{\ell-m+1}$, namely $Q^*\subseteq\omega^{\ell-m+1}$. Now, the refined subdomain hierarchy is a nested enlargement of the original one according to \eqref{eq:nested}, and, consequently, $\omega^{\ell-m+1}\subseteq\omega^{\ell-m+1,*}$, which implies $Q^*\subseteq\omega^{\ell-m+1,*}$.
\item If $Q^*\subseteq\Omega^{\ell,*}\setminus\Omega^\ell$, then Lemma~\ref{lma:rr} guarantees that $Q^*$ has been obtained by applying a single refinement to an element of $\cQ$. Hence, there exists $Q_{\cM}^{\#}\in\cG^{\ell-1}$ so that $Q_{\cM}^{\#}\supseteq Q^*$. Condition \ref{eq:sameshes} on $\cQ$ implies %{\Rd $Q_{\cM}^{\#}\subseteq \omega^{\ell-m}$}, with
\[
Q_{\cM}^{\#} 
{\Rd \,\subset\, {\omega}^{\ell-m}} = {\Rd \bigcup}
\left\{
{\Rd \overline{Q}\,:\,}
Q\in G^{\ell-m}
{\Rd \,\wedge\,}
S({\Rd Q},\ell-m)\subseteq
{\Omega}^{\ell-m}\right\}
\]
and, consequently,
\begin{equation}\label{eq:intermediate}
S({Q}_{\cM}^{\#},\ell-m)\subseteq
\Omega^{\ell-m}.
\end{equation}
%Since $\tilde{Q}'^{\ell-m+1}\subseteq \tilde{Q}'^{\ell-m}$, we have $\tilde{Q}'^{\ell-m+1}\subseteq \Omega^{\ell-m}$.
Since $Q_{\cM}^{\#}$ is an active element of $\cQ$ that has been subdivided in the refinement process from $\cQ$ to $\cQ^*$, the REFINE$\_$RECURSIVE module has been called over this element. More precisely, the call REFINE$\_$RECURSIVE $(\cQ^{\#},Q_{\cM}^{\#},m)$ belongs to the chain of recursive calls activated by REFINE$\_$RECURSIVE $(\cQ,Q_{\cM},m)$ for some intermediate mesh $\cQ^{\#}$ so that $\cQ^*\succeq\cQ^{\#}\succeq\cQ$. This mean that {\Rd the} recursive routine has been called on any $Q'\in\cN(\cQ^{\#},Q_{\cM}^{\#},m)$ with
\[
\cN(\cQ^{\#},Q_{\cM}^{\#},m) 
= \left\{Q'\in{\cG}^{\ell-m,\#}:
{\Rd \exists\, Q'' } \in
S({Q}_{\cM}^{\#},\ell-m+1) 
{\Rd , Q''\subseteq Q'}\right\}.
\]
By combinining \eqref{eq:intermediate} with $S({Q}_{\cM}^{\#},\ell-m+1)\subseteq S({Q}_{\cM}^{\#},\ell-m)$, we obtain $S({Q}_{\cM}^{\#},\ell-m+1)\subseteq 
{\Omega}^{\ell-m}$. Hence, the coarsest elements in $S({Q}_{\cM}^{\#},\ell-m+1)$ are exactly the ones of level $\ell-m$. All these $Q'$ elements of level $\ell-m$ have been subdivided into their children of level $\ell-m+1$ in the refinement step from $\cQ^{\#}$ to $\cQ^*$ in order to guarantee that
\[
S({Q}_{\cM}^{\#},\ell-m+1)\subseteq
\Omega^{\ell-m+1,*}.
\]
Then $Q^*\subseteq\omega^{\ell-m+1,*}$.
\end{itemize}
{In both cases, $Q^*\subseteq\omega^{\ell-m+1,*}$ implies $\widehat{Q}^*\subseteq\widehat{\omega}^{\ell-m+1,*}$. Condition \eqref{eq:sameshes} is then satisfied.}
\end{proof}

The previous results guarantees that the strict class of admissibility of the mesh is preserved by the REFINE$\_$RECURSIVE module. This result extends to the REFINE procedure.

\begin{crl}\label{crl:refine}
Let $\cQ$ be a strictly admissible mesh of class $m\ge 2$ and $\cM$ the set of elements of $\cQ$ marked for refinement. The call to $\cQ^* =$ REFINE $(\cQ,\cM,m)$ terminates and returns a  strictly admissible mesh $\cQ^*\succeq\cQ$ of class $m$. 
\end{crl}
\begin{proof} 
The termination of the REFINE module is directly implied by Lemma ~\ref{lma:rr}. 
Since every marked element $Q$ activates a call to REFINE$\_$RECURSIVE (${\cal Q},Q,m$), in order to prove that the final refined mesh ${\cal Q}^*$ preserves the satisfation of \eqref{eq:sameshes} and, consequently, the class $m$ of admissibility of ${\cal Q}$, it is sufficient to prove that this property holds after every recursive call. This is guaranteed by Proposition~\ref{prn:rr}.
\end{proof}

In view of the above corollary, we know that the refine mesh $\cQ^*$ preserves the class of admissibility of the initial mesh $\cQ$ and, consequently, Corollaries~\ref{crl:thbcor1} and \ref{crl:thbcor2} hold.

\begin{exm}\label{exm02}
An example for the case $m=2$  {and the identity map} is shown Figure~\ref{fig:exm03}.% and \ref{fig:exm02b}
\end{exm}

\begin{figure}[ht!]\begin{center}
\hspace*{-.75cm}
\subfigure[initial mesh]{
\includegraphics[scale=0.25]{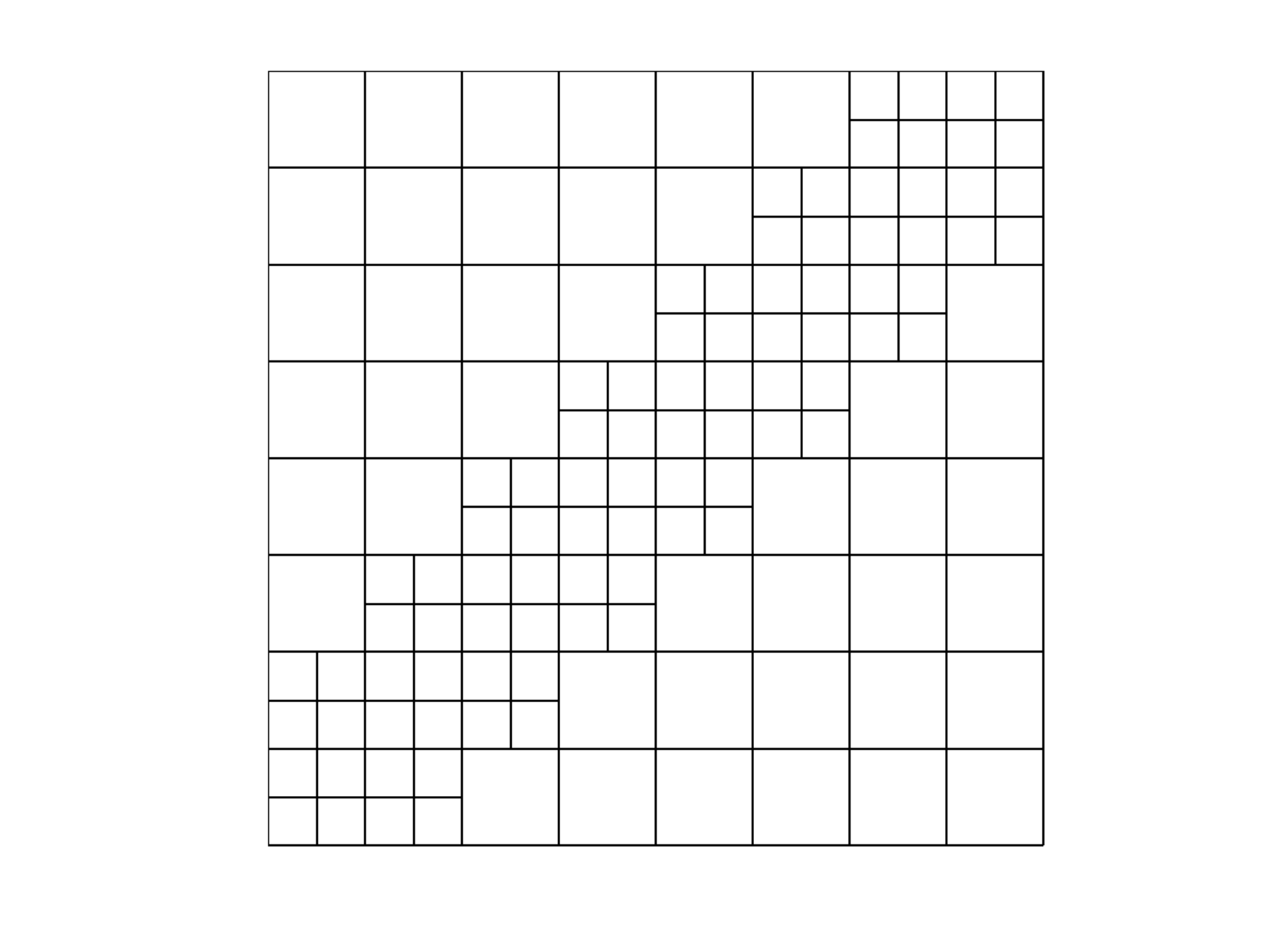}}
\hspace*{-.75cm}
\subfigure[marked elements]{
\includegraphics[scale=0.25]{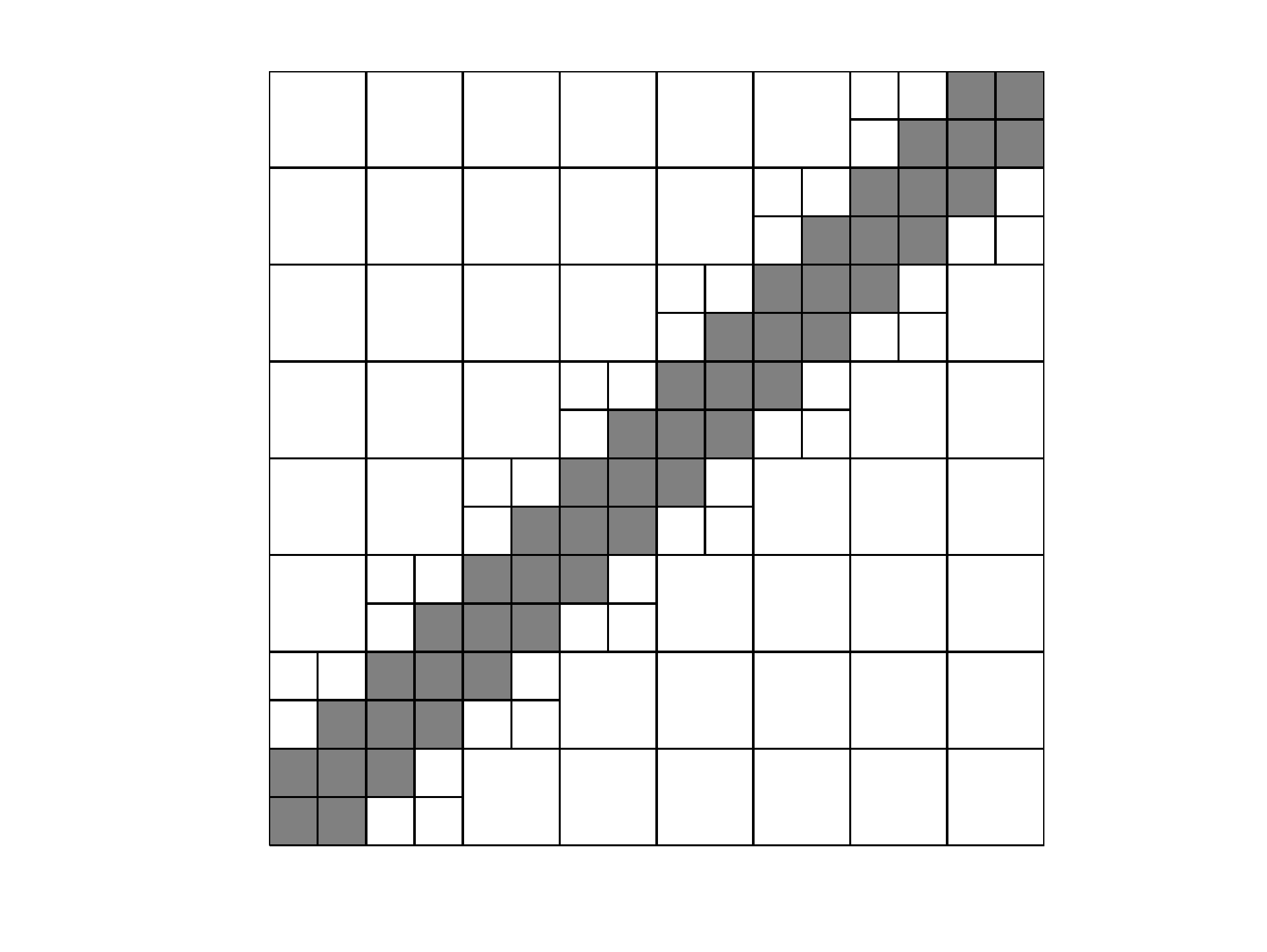}}
\caption{The admissible meshes in Fig.~\ref{fig:exm01} are generated by the call to $\cQ^* =$ REFINE $(\cQ,\cM,2)$ to the mesh $\cQ$ depicted in (a) with the marked set $\cM$  of elements shown in (b).}
\label{fig:exm03}
\end{center}\end{figure}
%------------------------------------------------------------------
\subsection{Contraction of the quasi-error and convergence}
%\label{sec:hspaces}
%------------------------------------------------------------------

Following the approach by \cite{cascon2008} (see also \cite{NochettoCIME}), we can  prove 
the contraction of the \emph{quasi-error}, defined as the contribution given by the energy error together with the estimator scaled by a positve factor $\gamma$:
\[
|||u-U|||_{\Omega}^2
+
\gamma\,\varepsilon^2_{\cQ}(U,\cQ)
\]
where, we remind  the energy norm is just  $||| \cdot |||_\Omega =  a(\cdot,\cdot)^{1/2}$ as defined in Section \ref{sec:solve}.
Note that neither the energy error $|||u-U|||_{\Omega}$, nor the estimator $\varepsilon_{\cQ}(U,\cQ)$  considered alone may satisfy a similar contraction property between two consecutive steps of the adaptive procedure in the general setting \cite{NochettoCIME}. 

In  the case of adaptive finite elements, monotonicity of the error is proved only under additional assumptions, see e.g., \cite{morin2001} and \cite{morin2002}, and indeed, we will not study this property in the present paper.

{We present here only the statement of the contraction theorem. Since its proof follows the analogous one for finite elements with only minor changes, we postpone it to the Appendix.}

 %\marginlabel{\Gd si da noia con l'$\alpha$ della (10) o forse no? NO DAI\B}
\begin{thm}\label{thm:contraction}
Let $\theta\in (0,1]$ be the D\"orfler marking parameter introduced in 
\eqref{eq:dm}, and let $\{\cQ_k,\SS(\cQ_k),U_k\}_{k\ge0}$ be the sequence
of strictly admissible meshes, hierarchical spline spaces, and discrete solution computed
by the adaptive procedure for the model problem \eqref{eq:mp}.
Then, there exist $\gamma>0$ and $0<\alpha<1$, independent of $k$ such that for all $k>0$ it holds:
\begin{equation}\label{eq:contraction}
|||u-U_{k+1}|||_{\Omega}^2+\gamma\,\varepsilon^2_{\cQ_{k+1}}(U_{k+1},\cQ_{k+1})\le
\alpha^2\left[|||u-U_k|||_{\Omega}^2
+\gamma\,\varepsilon^2_{\cQ_k}(U_{k},\cQ_k)\right].
\end{equation}
\end{thm}

An immediate consequence of this theorem, it the convergence of the error and of the estimator: 

\begin{crl}
  \label{crl:conv}
Under the same assumption of Theorem \ref{thm:contraction}, both the error and the estimator converge geometrically to $0$. I.e., 
there exists { $\gamma>0$, $0<\alpha<1$} and a constant $M$ such that 
\[ |||u-U_{k+1}|||_{\Omega}+\gamma\,\varepsilon_{\cQ_{k+1}}(U_{k+1},\cQ_{k+1}) \leq M \alpha^{k}, \]
where $M$ depends  on the bounds \eqref{eq:5} and \eqref{eq:Fbound}, but not on $k$. 
\end{crl}

%A few remarks are now due. 

\begin{rmk}\label{rmk:mark}
The questions related to convergence for other marking strategies remain open and may require additional assumptions on the refinement module which should be further investigated. On the other hand, it seems very plausible that any other error estimator verifying the upper bound provided by Theorem ~\ref{thm:ub} could be used to replace the simple residual based error indicator  proposed in this paper. As a side remark, we also note that the proof of Theorem~\ref{thm:contraction} does not require the lower bound presented in Theorem~\ref{thm:lb}. 
\end{rmk}
%------------------------------------------------------------------
\section{Closure}
\label{sec:closure}
%------------------------------------------------------------------
A posteriori residual-type estimators for the error associated to the Galerkin approximation of a simple model problem have been presented, based on the truncated basis for hierarchical splines with respect to some class of admissible meshes and a certain multilevel refinement. In the case of the upper bound, two key properties of the basis are exploited together with standard inequalities of (adaptive) finite element methods. First, the partition of unity property and, second, the bound for the number of basis functions that assume non-zero value on any mesh element. In the case of the lower bound, classical arguments of finite element estimates can be directly applied. By taking into account the a posteriori upper bound previously computed (ESTIMATE) and a classical marking strategy (MARK), we introduce a specific refinement procedure (REFINE) to proof the contraction of the quasi-error and, consequently, the convergence of the adaptive isogeometric methods.  

Corollary \ref{crl:conv} states convergence, but complexity is not analyzed at this stage. In other words, we have not proven any connection between the error and the number of degrees of freedom that are needed to compute the iterate $U^k$. As it is known from the AFEM theory, this  can be studied by analyzing the complexity of REFINE. In order to do this, we need to understand how the refinement module controls the interplay between the number of refined elements $\# \cQ_k-\# \cQ_0$ introduced up to step $k$ (that influences the degrees of freedom added during the refinement) and the total number of marked elements. Among other things, an estimate of the type: there exists a certain constant $\Lambda_0>0$ such that
\[ \# \cQ_k - \# \cQ_0 \leq \Lambda_0  \sum_{j=0}^{k-1} \# \cM_j\]
is in need. This kind of complexity estimate has been derived for adaptive finite elements in \cite{binev2004a,stevenson2007} for two- and three-dimensional problems, respectively.  We will prove an analogous estimate for the adaptive isogeometric method here introduced, together with optimal convergence rates, in  the companion paper \cite{buffa2015b}.

%\Rd Finally,  if the extension of our adaptive framework to general analysis suitable T-splines (or other \Gd ARGH.. come dirlo ??\Rd)  seems to be a long term objective, the use  of analysis-suitable T-splines combined with semi--structured hierarchical construction has been recently investigated \cite{scott2014a,evans2015}. The possibility of extending the adaptivity theory here presented to this case is a challenging issue, but the wide modeling capabilities of T-splines encapsulated into the hierarchical model would provide a powerful refine module.\B

 Suitable extensions of our adaptive framework may be investigated in order to consider less restrictive mesh configurations. For example, the use  of analysis-suitable T-splines combined with semi--structured hierarchical construction has been recently investigated \cite{scott2014a,evans2015}. The possibility of extending the adaptivity theory here presented to this case is a challenging issue, but the wide modeling capabilities of T-splines encapsulated into the hierarchical model would provide a powerful refine module.\B
\section*{Acknowledgment}
This work was supported by the Gruppo Nazionale per il Calcolo Scientifico 
(GNCS) of the Istituto Nazionale di Alta Matematica (INdAM) 
and by the project DREAMS (MIUR ``Futuro in Ricerca'' RBFR13FBI3).
%------------------------------------------------------------------
\appendix
\section*{Appendix}
\label{sec:appendix}
%------------------------------------------------------------------
This appendix is devoted to the proof of Theorem \ref{thm:contraction}. Here we basically reproduce the proof of the statement as it was first proved for finite elements. In our presentation, we closely follow Chapter 5 of \cite{NochettoCIME}. We are not adding something new, and the value of this appendix is only show that the same arguments used for finite elements apply also to our case with very minor changes. Indeed, some of the proofs are made easier by the fact that our error indicator does not contain jump terms. 

Before proving Theorem \ref{thm:contraction}, we need a few preparatory lemmas.

This first Lemma is nothing else then the Pytaghoras theorem. See Lemma 12 in \cite{NochettoCIME}.

\begin{lma}\label{lma:pyt}
Let $\cQ$ be an admissible mesh and $\cQ^*$ be a refinement of $\cQ$, i.e., $ {\cal Q^*}\succeq{\cal Q}$. Let $U$ and $U^*$ be the Galerkin solution of problem \eqref{eq:gal} on $\mathbb{S}_D(\cQ)$ and  $\mathbb{S}_D(\cQ^*)$, respectively. It holds:
\begin{equation}\label{eq:pyt}
|||u-U^*|||_{\Omega}^2 = |||u-U|||_{\Omega}^2-|||U^*-U|||_{\Omega}^2.
\end{equation}%\marginlabel{$\Rightarrow$ use ${\cal Q}/{\cal Q}^*$ or ${\cal Q}_k/{\cal Q}_{k+1}$?...}
\end{lma}
\begin{proof}
This is an immediate consequence of the Galerkin orthogonality.
\end{proof}

The next Lemma provides a measure of the reduction of the error indicator $\varepsilon_\cQ(U)$ with respect to the mesh $\cQ$,  
see Lemma 13 in \cite{NochettoCIME}.

\begin{lma}\label{lma:erred}
Let $\cQ$ be a stricty admissible mesh \B, $\cM$ be a set of marked elements and $\cQ^*$ the corresponding refined mesh. i.e., $\cQ^*= \mathrm{REFINE}(\cQ,\cM)$. 
 Then, for all $V\in \mathbb{S}_D(\cQ)$ it holds
$\forall V\in\mathbb{V}({\cal Q})$,
\begin{equation}\label{eq:erred}
\varepsilon_{{\cal Q}^*}^2(U,{\cal Q}^*)\le
\varepsilon_{{\cal Q}}^2(U,{\cal Q})
-\lambda\varepsilon_{{\cal Q}}^2(U,\cM)
\end{equation}
 where $0< \lambda <1$.  \B
%where $\lambda=1-1/q^2$, with $q=2$ for dyadic refinement and $q=\sqrt[d]{\prod_{j=1}^d p_j}$ for p-adic refinement.  
\end{lma}

\begin{proof}
 For each $Q\in \cM$, we denote by $\cQ^*(Q)$ the collection of elements of $\cQ^*$ that are created by splitting $Q$. We know that, by construction,  for all $Q^\star\in   \cQ^*(Q)$, it holds  $h_{\hat{Q}^*} \leq \displaystyle \frac{1}{2} \; h_{\hat{Q}}$. Due to \eqref{eq:Fbound}, there exists then a constant $c(\bF)$, $c(\bF) < 1$,  independent of $Q$ such that 
$h_{Q^\star} \leq c(\bF) h_Q$. \B 
 
If we adopt the notation \[\varepsilon^2_{\cQ^*} (U,Q)= \sum_{Q^*\in  \cQ^*(Q)} h_{Q^*}^2 \| r(U) \|^2_{L^2(Q^*)},\]
 it clearly holds:
\[ \forall Q\in \cM \qquad  \varepsilon_{\cQ^*} (U,Q) \leq  c(\bF) \B \; \varepsilon_{\cQ} (U,Q). \]
Moreover, since the mesh size does not increase for all elements in $\cQ\setminus \cM$, we have:
\[   \forall Q\in \cQ\setminus \cM \qquad  \varepsilon_{\cQ^*} (U,Q) \leq  \varepsilon_{\cQ} (U,Q).\]
Summing up for all $Q\in \cQ$, we obtain:
\[   \varepsilon^2_{\cQ^*} (U,Q^*) \leq  \varepsilon^2_{\cQ} (U,\cQ\setminus\cM) +  c^2(\bF) \B \;   \varepsilon_{\cQ} (U,\cM)\]
which implies then \eqref{eq:erred}  with $\lambda = 1 - c^2(\bF)$.\B 
\end{proof}

We turn now to the Lipschitz property of the error indicator $\varepsilon_\cQ(U,Q)$, for any $Q$,  with respect to the trial function $U$, see Lemma 14 in \cite{NochettoCIME}.

\begin{lma}\label{lma:lip}
Let $\cQ$ be an admissible mesh and $V\,,\ W\in \mathbb{S}_D(\cQ)$.   There exists a $\Lambda>0$ such that the following holds for all $ Q\in{\cal Q}$
\begin{equation}\label{eq:lip}
|\varepsilon_{\cal Q}(V,Q)-\varepsilon_{\cal Q}(W,Q)|\leq \Lambda 
\eta_{\cal Q}(\bA,Q)||\nabla(V-W)||_{L^2(Q)},
\end{equation}
where  $\eta_{\cal Q}(\bA,Q)= h_Q \| \mathrm{div}(\bA)\|_{L^\infty(Q)} + \| \bA\|_{L^\infty(Q)}.$
\end{lma}

\begin{proof}
  By definition, we have:
\[ r(V) - r(W)= \mathrm{\div}(\bA\nabla (V-W )) = \mathrm{\div}(\bA) \cdot \nabla(V-W) + \bA: D^2(U-W),  \]
where $D^2\cdot$ stands for the hessian matrix. 
 Using the inverse inequality $\| D^2(U-W) \|_{L^2(Q)} \les h_Q^{-1}  \| \nabla(V-W)\|_{L^2(Q)}$,  applying Cauchy-Schwarz  and triangle inequality, we obtain:
\begin{equation*}
  \begin{aligned}
    |\varepsilon_{\cal Q}(V,Q)-\varepsilon_{\cal Q}(W,Q)| & \les h_Q \| r(V) - r(W) \|_{L^2(Q)} \\
& \les (h_Q \| \mathrm{div}(\bA)\|_{L^\infty(Q)} + \| \bA\|_{L^\infty(Q)}) \| \nabla(V-W)\|_{L^2(Q)} 
  \end{aligned}
\end{equation*}
which ends the proof. 
\end{proof}

We can combine the previous results to obtain the last preparatory Lemma. See Proposition 3 in \cite{NochettoCIME}.

\begin{lma}\label{lma:esred}
Let $\cQ$ be a  strictly \B  admissible mesh, $\cM$ be a set of marked elements and $\cQ^*$ the corresponding refined mesh. i.e., $\cQ^*= \mathrm{REFINE}(\cQ,\cM)$.  There exists $\Lambda>0$ so that, $\forall V\in\mathbb{S}_D({\cal Q}), V^*\in\mathbb{S}_D^*({\cal Q}^*)$ and any $\delta>0$,
\begin{equation}\label{eq:esred}
\varepsilon^2_{{\cal Q}^*}(V^*,{\cal Q}^*)\le
(1+\delta)\left[\varepsilon^2_{\cal Q}(V,{\cal Q})
-\lambda\varepsilon^2_{\cal Q}(V,{\cal M})\right]
+(1+\delta^{-1})\Lambda^2\eta^2_{\cal Q}(A,{\cal Q})
|||V^*-V|||_{\Omega}^2
\end{equation} 
   with $\eta^2_{\cal Q^*} = \sup_{Q^*\in{\cal Q}^*}\eta^2_{\cal Q^*}(\bA,Q^*)$. \B 
\end{lma}

\begin{proof}
  Applying triangle inequality and Lemma \ref{lma:lip}, we have:
\begin{equation*}
  \begin{aligned}
\varepsilon^2_{{\cal Q}^*}(V^*,{Q}^*) &\leq (1+\delta) \varepsilon^2_{\cQ^*}(V,{Q^*}) + (1+\delta^{-1}) \, | \varepsilon_{{\cal Q}^*}(V^*,{Q}^*) - \varepsilon_{\cQ^*}(V,{Q^*}) |^2 \\
& \leq (1+\delta) \varepsilon^2_{\cQ^*}(V,{Q^*})  + \eta^2_{\cal Q^*}(A,Q^*) \Lambda \| \nabla(V-V^*)\|^2_{L^2(Q^*)}. 
 \end{aligned}
\end{equation*}
Summing over the elements, we obtain:
\[  \varepsilon^2_{{\cal Q}^*}(V^*,{\cal Q}^*) \leq (1+\delta) \varepsilon^2_{\cQ^*}(V,{\cQ^*})  + \eta^2_{\cal Q^*}  \|(V-V^*)\|^2_{\mathbb{V}}. \]
The statement follows by applying Lemma \ref{lma:erred}. 
\end{proof}
 %%%%%%%

Finally we are now ready to  prove Theorem \ref{thm:contraction}.

\begin{proofof}{Theorem \ref{thm:contraction}}
By summing up the error orthogonality \eqref{eq:pyt} with the estimator reduction
\eqref{eq:esred} scaled by a constant $\gamma>0$, we obtain
\begin{align*}
|||u-U_{k+1}|||_{\Omega}^2+\gamma\,\varepsilon^2_{\cQ_{k+1}}(U_{k+1},\cQ_{k+1})
\,\le\, &
|||u-U_{k}|||_{\Omega}^2 \\
\,+\, &
\left[\gamma\,(1+\delta^{-1})\Lambda_0-1\right] |||U_{k+1}-U_k|||_{\Omega}^2\\
\,+\, &
\gamma\,(1+\delta)\left[\varepsilon^2_{\cQ_k}(U_k,\cQ_k)
-\lambda\,\varepsilon^2_{\cQ_k}(U_k,\cM_k)\right],
\end{align*}
where we have used \eqref{eq:erred} with $\cQ=\cQ_k, \cQ^*=\cQ_{k+1}, V=U_k, V^*=U_{k+1}$, and we have set $\Lambda_0=
\Lambda \eta_{\cQ_0}^2(A,\cQ_0)\ge \Lambda \eta_{\cQ_k}^2(A,\cQ_k)$.
%\marginlabel{$\Lambda=1$ in \eqref{eq:esred}}
The choice $\gamma=1/[(1+\delta^{-1})\,\Lambda_0]$ together with the D\"orfler 
marking property \eqref{eq:dm} leads to
\begin{align*}
|||u-U_{k+1}|||_{\Omega}^2+\gamma\,\varepsilon^2_{\cQ_{k+1}}(U_{k+1},\cQ_{k+1})
\,\le\, &
|||u-U_{k}|||_{\Omega}^2 \\
\,+\, &
\gamma\,(1+\delta)\left[\varepsilon^2_{\cQ_k}(U_k,\cQ_k)
%-\lambda\,\varepsilon^2_{\cQ_k}(U_k,\cM_k)\right]
-\lambda\,\theta^2\,\varepsilon^2_{\cQ_k}(U_k,\cQ_k)\right]\\
\,=\, &
|||u-U_{k}|||_{\Omega}^2 + 
\gamma\,(1+\delta)(1-\lambda\,\theta^2)\,\varepsilon^2_{\cQ_k}(U_k,\cQ_k).
\end{align*}
By choosing the parameter $\delta$ so that $(1+\delta)(1-\lambda\,\theta^2)=1-\lambda\,\theta^2/2$, the above inequality reduces to
\[
|||u-U_{k+1}|||_{\Omega}^2+\gamma\,\varepsilon^2_{\cQ_{k+1}}(U_{k+1},\cQ_{k+1})
\,\le\,
|||u-U_{k}|||_{\Omega}^2 + 
\gamma\,\left(1-\frac{\lambda\,\theta^2}{2}\right)\,\varepsilon^2_{\cQ_k}(U_k,\cQ_k).
\]
The second term on the right-hand side may be written as
\[
-\gamma\,\frac{\lambda\,\theta^2}{4}\,\varepsilon^2_{\cQ_k}(U_k,\cQ_k)
+
\gamma\,\left(1-\frac{\lambda\,\theta^2}{4}\right)\,\varepsilon^2_{\cQ_k}(U_k,\cQ_k),
\]
so that taking into account the a posteriori upper bound \eqref{eq:ub} and the associated 
constant $C_{\mathrm{up}}$, we obtain the inequality \eqref{eq:contraction}
% \[
% |||u-U_{k+1}|||_{\Omega}^2+\gamma\,\varepsilon^2_{\cQ_{k+1}}(U_{k+1},\cQ_{k+1})
% \,\le\,
% \alpha^2\,\left[ |||u-U_{k}|||_{\Omega}^2 
% + 
% \gamma\, \varepsilon^2_{\cQ_k}(U_k,\cQ_k)\right],
% \]
with $\alpha=\max\left\{1-\gamma\frac{\lambda\,\theta^2}{4\,C_{\mathrm{up}}},1-\frac{\lambda\,\theta^2}{4}\right\}<1$.
\end{proofof}%\marginlabel{$C$ = constant for UP}
%------------------------------------------------------------------
%\bibliographystyle{abbrv}
%\bibliography{paper}
\bibliographystyle{plain}
\bibliography{biblio}

\begin{thebibliography}{10}

\bibitem{bazilevs2010}
Y.~Bazilevs, V.~M. Calo, J.~A. Cottrell, J.~Evans, T.~J.~R. Hughes, S.~Lipton,
  M.~A. Scott, and T.~W. Sederberg.
\newblock Isogeometric analysis using {T-Splines}.
\newblock {\em \CMAME}, 199:229--263, 2010.

\bibitem{binev2004a}
P.~Binev, W.~Dahmen, and R.~DeVore.
\newblock {Adaptive Finite Element Methods with convergence rates}.
\newblock {\em \NM}, 97:219--268, 2004.

\bibitem{bonito2010}
A.~Bonito and R.~H. Nochetto.
\newblock {Quasi-optimal convergence rate of an adaptive discontinuous Galerkin
  method}.
\newblock {\em \SIAMJNA}, 48:734--771, 2010.

\bibitem{bressan2013}
A.~Bressan.
\newblock {Some properties of LR-splines}.
\newblock {\em \CAGD}, 30:778--794, 2013.

\bibitem{buffa2015b}
A.~Buffa and C.~Giannelli.
\newblock Adaptive isogeometric methods with hierarchical splines: optimality
  and convergence rates.
\newblock In preparation, 2015.

\bibitem{cascon2008}
J.~M. Casc\'{o}n, C.~Kreuzer, R.~H. Nochetto, and K.~G. Siebert.
\newblock Quasi-optimal convergence rate for an adaptive finite element method.
\newblock {\em \SIAMJNA}, 46:2524--2550, 2008.

\bibitem{cottrell2009}
J.~A. Cottrell, T.~J.~R. Hughes, and Y.~Bazilevs.
\newblock {\em Isogeometric Analysis: Toward Integration of CAD and FEA}.
\newblock John Wiley \& Sons, 2009.

\bibitem{daveiga2012}
L.~B. da~Veiga, A.~Buffa, D.~Cho, and G.~Sangalli.
\newblock {Analysis-Suitable T-splines are Dual-Compatible}.
\newblock {\em \CMAME}, 249--252:42--51, 2012.

\bibitem{daveiga2013}
L.~B. da~Veiga, A.~Buffa, G.~Sangalli, and R.~V\'azquez.
\newblock {Analysis-suitable T-splines of arbitrary degree: definition, linear
  independence and approximation properties}.
\newblock {\em \M3AS}, 23:1979--2003, 2013.

\bibitem{deboor2001}
C.~de~Boor.
\newblock {\em A practical guide to splines}.
\newblock Springer, revised ed., 2001.

\bibitem{ds2012}
L.~Ded\`e and H.~A.~F.~A. Santos.
\newblock {B--spline goal-oriented error estimators for geometrically nonlinear
  rods}.
\newblock {\em \CM}, 49:35--52, 2012.

\bibitem{deng2006}
J.~Deng, F.~Chen, and Y.~Feng.
\newblock Dimensions of spline spaces over {T}--meshes.
\newblock {\em \JCAM}, 194:267--283, 2006.

\bibitem{deng2008}
J.~Deng, F.~Chen, X.~Li, Ch. Hu, W.~Tong, Z.~Yang, and Y.~Feng.
\newblock Polynomial splines over hierarchical {T-meshes}.
\newblock {\em \GM}, 70:76--86, 2008.

\bibitem{dokken2013}
T.~Dokken, T.~Lyche, and K.~F. Pettersen.
\newblock Polynomial splines over locally refined box--partitions.
\newblock {\em \CAGD}, 30:331--356, 2013.

\bibitem{doerfel2010}
M.~R. D\"orfel, B.~J\"uttler, and B.~Simeon.
\newblock Adaptive isogeometric analysis by local h-refinement with
  {T-splines}.
\newblock {\em \CMAME}, 199:264--275, 2010.

\bibitem{dorfler1996}
W.~D\"orfler.
\newblock A convergent algorithm for poisson's equation.
\newblock {\em \SIAMJNA}, 33:1106--1124, 1996.

\bibitem{evans2015}
E.~J. Evans, M.~A. Scott, X.~Li, and D.~C. Thomas.
\newblock {Hierarchical T-splines: Analysis-suitability, B\'ezier extraction,
  and application as an adaptive basis for isogeometric analysis}.
\newblock {\em \CMAME}, 284:1--20, 2015.

\bibitem{giannelli2013}
C.~Giannelli and B.~J\"uttler.
\newblock Bases and dimensions of bivariate hierarchical tensor--product
  splines.
\newblock {\em \JCAM}, 239:162--178, 2013.

\bibitem{giannelli2012}
C.~Giannelli, B.~J\"uttler, and H.~Speleers.
\newblock {THB}--splines: the truncated basis for hierarchical splines.
\newblock {\em \CAGD}, 29:485--498, 2012.

\bibitem{giannelli2014}
C.~Giannelli, B.~J\"uttler, and H.~Speleers.
\newblock Strongly stable bases for adaptively refined multilevel spline
  spaces.
\newblock {\em \ACM}, 40:459--490, 2014.

\bibitem{hughes2005}
T.~J.~R. Hughes, J.~A. Cottrell, and Y.~Bazilevs.
\newblock Isogeometric analysis: {CAD}, finite elements, {NURBS}, exact
  geometry and mesh refinement.
\newblock {\em \CMAME}, 194:4135--4195, 2005.

\bibitem{johannessen2014}
K.~A. Johannessen, T.~Kvamsdal, and T.~Dokken.
\newblock {Isogeometric analysis using LR B-splines}.
\newblock {\em \CMAME}, 269:471--514, 2014.

\bibitem{kiss2014b}
G.~Kiss, C.~Giannelli, U.~Zore, B.~J\"uttler, D.~Gro{\ss}mann, and J.~Barner.
\newblock {Adaptive CAD model (re--)construction with THB--splines}.
\newblock {\em Graphical models}, 76:273--288, 2014.

\bibitem{kraft1997}
R.~Kraft.
\newblock Adaptive and linearly independent multilevel {B}--splines.
\newblock In A.~Le~M{\'e}haut{\'e}, C.~Rabut, and L.~L. Schumaker, editors,
  {\em Surface Fitting and Multiresolution Methods}, pages 209--218. Vanderbilt
  University Press, Nashville, 1997.

\bibitem{kvvb2014}
G.~Kuru, C.~V. Verhoosel, K.~G. van~der Zeeb, and E.~H. van Brummelen.
\newblock Goal-adaptive isogeometric analysis with hierarchical splines.
\newblock {\em \CMAME.}, 270:270--292, 2014.

\bibitem{mokris2014a}
D.~Mokri\v{s}, B.~J\"uttler, and C.~Giannelli.
\newblock On the completeness of hierarchical tensor-product {B}-splines.
\newblock {\em \JCAM}, 271:53--70, 2014.

\bibitem{morin2001}
P.~Morin, R.~H. Nochetto, and K.~G. Siebert.
\newblock {Data oscillation and convergence of adaptive FEM}.
\newblock {\em \SIAMJNA}, 38:466--488, 2001.

\bibitem{morin2002}
P.~Morin, R.~H. Nochetto, and K.~G. Siebert.
\newblock {Convergence of adaptive finite element methods}.
\newblock {\em SIAM Rev.}, 44:631--658, 2002.

\bibitem{nguyen-thanh2011}
N.~Nguyen-Thanh, H.~Nguyen-Xuan, S.~P.~A. Bordas, and T.~Rabczuk.
\newblock Isogeometric analysis using polynomial splines over hierarchical
  {T}-meshes for two-dimensional elastic solids.
\newblock {\em \CMAME}, 200:1892--1908, 2011.

\bibitem{nsv2009}
R.~H. Nochetto, K.~G. Siebert, and A.~Veeser.
\newblock Theory of adaptive finite element methods: An introduction.
\newblock In Ronald DeVore and Angela Kunoth, editors, {\em Multiscale,
  Nonlinear and Adaptive Approximation}, pages 409--542. Springer Berlin
  Heidelberg, 2009.

\bibitem{NochettoCIME}
R.~H. Nochetto and A.~Veeser.
\newblock Primer of adaptive finite element methods.
\newblock In {\em Multiscale and adaptivity: modeling, numerics and
  applications}, volume 2040 of {\em Lecture Notes in Math.}, pages 125--225.
  Springer, Heidelberg, 2012.

\bibitem{schumaker2007}
L.~L. Schumaker.
\newblock {\em {Spline Functions: Basic Theory, 3rd Edition}}.
\newblock Cambridge University Press, 2007.

\bibitem{scott2011b}
M.~A. Scott, X.~Li, T.~W. Sederberg, and T.~J.~R. Hughes.
\newblock Local refinement of analysis-suitable {T}-splines.
\newblock {\em \CMAME}, 213--216:206--222, 2012.

\bibitem{scott2014a}
M.~A. Scott, D.~C. Thomas, and E.~J. Evans.
\newblock Isogeometric spline forests.
\newblock {\em \CMAME}, 269:222--264, 2014.

\bibitem{sederberg2004}
T.~W. Sederberg, D.~L. Cardon, G.~T. Finnigan, N.~S. North, J.~Zheng, and
  T.~Lyche.
\newblock {T}-spline simplification and local refinement.
\newblock {\em \ACMTG}, 23:276 -- 283, 2004.

\bibitem{sederberg2003}
T.~W. Sederberg, J.~Zheng, A.~Bakenov, and A.~Nasri.
\newblock {T}-splines and {T-NURCCS}.
\newblock {\em \ACMTG}, 22:477--484, 2003.

\bibitem{sm2015}
H.~Speleers and C.~Manni.
\newblock Effortless quasi-interpolation in hierarchical spaces.
\newblock {\em \NM}, to appear, 2015.

\bibitem{stevenson2007}
R.~Stevenson.
\newblock Optimality of a standard adaptive finite element method.
\newblock {\em \FCM}, 7:245--269, 2007.

\bibitem{stevenson2008}
R.~Stevenson.
\newblock The completion of locally refined simplicial partitions created by
  bisection.
\newblock {\em \MC}, 77:227--241, 2008.

\bibitem{zv2011}
K.~G. van~der Zee and C.~V. Verhoosel.
\newblock Isogeometric analysis-based goal-oriented error estimation for
  free-boundary problems.
\newblock {\em \FEAD}, 47:600--609, 2011.

\bibitem{verfurth2013}
R.~Verf\"urth.
\newblock {\em A Posteriori Error Estimation Techniques for Finite Element
  Methods}.
\newblock Oxford University Press, 2013.

\bibitem{vuong2011}
A.-V. Vuong, C.~Giannelli, B.~J\"uttler, and B.~Simeon.
\newblock A hierarchical approach to adaptive local refinement in isogeometric
  analysis.
\newblock {\em \CMAME}, 200:3554--3567, 2011.

\end{thebibliography}

%\section*{References}
%They are to be cited in the text in superscript  
%after comma and period (e.g.~word,\cite{am})   
%but before other punctuation marks like colons,  
%(e.g.~word\cite{lar}:) semi-colons and\break
%question marks. If it is mentioned in the text as part of a sentence, 
%it should be of normal size, e.g.~see Ref.~\refcite{lar}.
%
%\begin{thebibliography}{00}
%%1
%\bibitem{aiz} M. Aizenman and T. Bak, Convergence to equilibrium
%in a system of reacting polymers, {\it Comm. Math. Phys.} {\bf 65}
%(1979) 203--230.
%
%\end{thebibliography}

\end{document}